\documentclass[11pt,draftcls,onecolumn,twoside]{IEEEtran} %\documentclass{IEEEtran}
\usepackage{epsfig}
\usepackage{amsmath}
\usepackage{amstext}
\usepackage{amsfonts}
\usepackage{amssymb}
\usepackage{eucal}
\usepackage{graphicx}
\usepackage{verbatim}
\usepackage{stmaryrd}
\usepackage{pstricks}
\usepackage{pst-node}
\usepackage{psfrag}
\usepackage{bm}
\usepackage[latin1]{inputenc}
\usepackage{tikz}
\usepackage{pgfplots}
\usepackage{algorithmic}

\newcommand{\cov}{\mathrm{cov}}
\newcommand{\tr}{\mathrm{Tr}\,}

\newcommand{\OO}{\mathcal O}

\newtheorem{theorem}{Theorem}
\newtheorem{lemma}{Lemma}
\newtheorem{proposition}{Proposition}

\newtheorem{remark}{Remark}

\begin{document}
\bibliographystyle{IEEEtran}
\title{Fluctuations of an Improved Population Eigenvalue Estimator in Sample Covariance Matrix Models}
\author{J. Yao, R. Couillet, J. Najim, M. Debbah\\
\today
\thanks{This work was partially supported by Agence Nationale de la Recherche (France), program ANR-07-MDCO-012-01 'Sesame'.}
\thanks{Yao and Najim are with T\'el\'ecom Paristech, France. {\tt \{yao,najim\}@telecom-paristech.fr}\ ;}
\thanks{Yao is also with Ecole Normale Sup\'erieure and Najim, with CNRS.}
\thanks{Couillet and Debbah are with Sup\'elec, France. {\tt \{romain.couillet, merouane.debbah\}@supelec.fr},}
\thanks{Debbah also holds Alcatel-Lucent/Sup\'elec Flexible Radio chair,
France\emph{}.}
}

\maketitle

\begin{abstract}
  This article provides a central limit theorem for a consistent
  estimator of population eigenvalues with large multiplicities based
  on sample covariance matrices. The focus is on limited sample size situations, whereby the number of available observations is known and comparable in magnitude to the observation dimension. An exact expression as well as an
  empirical, asymptotically accurate, approximation of the limiting
  variance is derived. Simulations are performed that corroborate the theoretical claims. A specific application to wireless sensor networks is developed.
\end{abstract}

\section{Introduction}

Problems of statistical inference based on $M$ independent
observations of an $N$-variate random variable $\bf y$, with
$\mathbb{E}[{\bf y}]=0$ and $\mathbb{E} [{\bf y} {\bf y}^H]={\bf{R}}_N$
have drawn the attention of researchers from many fields for years:
Portfolio optimization in finance \cite{PLE02}, gene coexistence in
biostatistics \cite{LUO07}, channel capacity in wireless communications
\cite{Abla09}, power estimation in sensor networks \cite{COU10b}, array processing \cite{MES08c}, etc.

In particular, retrieving spectral properties of the {\it population covariance
  matrix} ${\bf R}_N$, based on the observation of $M$ independent and identically distributed (i.i.d.) samples ${\bf
  y}^{(1)},\ldots,{\bf y}^{(M)}$, is paramount to many
questions of general science. If $M$ is large compared to $N$, then it
is known that almost surely $\Vert \hat{\bf R}_N-{\bf R}_N\Vert \to 0$,
as $M\to\infty$, for any standard matrix norm, where $\hat{\bf R}_N$
is the {\it sample covariance matrix} ${\hat{\bf R}_N} \triangleq
\frac1M\sum_{m=1}^M {\bf y}^{(m)}{\bf y}^{(m)H}$. However, one cannot
always afford a large number of samples, especially in wireless
communications where the number of available samples has often a
size comparable to the dimension of each sample. In order to cope with
this issue, random matrix theory \cite{SIL06,COUbook} has proposed new
tools, mainly spurred by the {\it G-estimators} of Girko
\cite{GIR00}. Other works include convex optimization methods
\cite{SIL92,KAR08} and free probability tools \cite{RYA07,COU08}. Many
of those estimators are consistent in the sense that they are
asymptotically unbiased as $M,N$ grow large at the same
rate. Nonetheless, only recently have techniques been unveiled which
allow to estimate individual eigenvalues and functionals of
eigenvectors of $\bf R$. The main contributor is Mestre
\cite{MES08}-\cite{MES08b} who studies the case where ${\bf R}_N = {\bf U}_N {\bf D}_N {\bf U}_N^{H}$ with ${\bf D}_N$ diagonal with entries of large multiplicities and ${\bf U}_N$ with i.i.d. entries. For this model, he provides an estimator for every eigenvalue of $\bf R$ with large multiplicity under some separability
condition, see also Vallet {\it et al.}  \cite{VAL09}, Couillet
{\it et al.} \cite{COU10b} for more elaborate models.

% If the entries of $\bf{y}$ are the monthly market evolutions of $N$
% retail products, then the largest eigenvalue and corresponding
% eigenvector of $\bf R$ characterize the optimal portfolio for a trader
% \cite{PLE02}. If $\bf y$ is the sample of alleles of $N$ genes
% extracted from a living being, then $\bf R$ predicts gene coexistence
% \cite{LUO07}. In wireless communications, if $\bf y$ are signals
% transmitted through a multi-dimensional channel, then the eigenvalues
% of $\bf R$ are a sufficient statistic for the capacity of this channel
% \cite{FOS98}. In the context of sensor networks, if $\bf{y}$ is a
% vector of data observed by $N$ sensors and arising from signals
% transmitted by $K$ sources with respective transmit powers
% $P_1,\ldots,P_K$, then the eigenvalues $\bf{R}$ contain information
% about those $P_k$. The present work will use this last scenario as an
% application example.

These estimators, although proven asymptotically unbiased, have
nonetheless not been fully characterized in terms of performance
statistics. It is in particular fundamental to evaluate the variance
of these estimators for not-too-large $M,N$. The purpose of this
article is to study the fluctuations of the population eigenvalue
estimator of \cite{MES08b} in the case of structured population covariance matrices.
A central limit theorem (CLT) is provided to describe the asymptotic fluctuations
of the estimators with exact expression for the variance as $M,N$ tend to infinity . An empirical
approximation, asymptotically accurate is also derived.

The results are applied in a cognitive radio context in which we assume the co-existence of a
licensed (primary) network and an opportunistic (secondary) network
aiming at reusing the bandwidth resources left unoccupied by the
primary network. The eigenvalue estimator is used here by secondary
users to estimate the transmit power of primary users, while the
fluctuations are used to provide a confidence margin on the estimate.

% In the context of sensor networks with collocated sensors and distant
% sources, evaluating the source transmit powers allows one to
% characterize the distance of the individual sources to the sensor
% network. A recent application in telecommunication, and more
% specifically cognitive radios \cite{MIT99}, is the blind evaluation of
% the distance of multiple {\it primary} sources; these distances are
% estimated by a {\it secondary} network, aiming at determining the
% optimal transmission coverage within the secondary network that
% ensures both a low probability of interference towards the primary
% network and high rate communications for the secondary users. If
% statistical information about the distance estimator is known to the
% secondary network, both requirements can be clearly optimized
The remainder of the article is structured as follows: In Section \ref{sec:model}, the
system model is introduced and the main results from
\cite{MES08,MES08b} are recalled. In Section \ref{sec:CLT}, the CLT
for the estimator in \cite{MES08b} is stated with the asymptotic
variance. In Section \ref{sec:estimation-cov}, an empirical
approximation for the variance is derived. A cognitive radio
application of these results is provided in Section
\ref{sec:application-cognitive}, with comparative Monte Carlo
simulations. Finally, Section \ref{sec:conclusion} concludes this
article. Technical proofs are postponed to the appendix.

\section{Estimation of the population eigenvalues}
\label{sec:model}

\subsection{Notations}
In this paper, the notations $s,{\bf x}, {\bf M}$ stand for scalars,
vectors and matrices, respectively. As usual, $\| {\bf x}\|$
represents the Euclidean norm of vector ${\bf x}$ and $\|{\bf M}\|$
stands for the spectral norm of ${\bf M}$. The superscripts
$(\cdot)^T$ and $(\cdot)^H$ respectively stand for the transpose and
transpose conjugate; the trace of ${\bf M}$ is denoted by
$\mathrm{Tr}( {\bf M})$; the mathematical expectation operator, by
$\mathbb{E}$. If ${\bf x}$ is a $N\times 1$ vector, then $\mathrm{diag}({\bf
  x})$ is the $N\times N$ matrix with diagonal elements the components
of ${\bf x}$.  If $z\in \mathbb{C}$, then $\Re(z)$ and $\Im(z)$
respectively stand for $z$'s real and imaginary parts, while
$\mathbf{i}$ stands for $\sqrt{-1}$; $\overline{z}$ stands
for $z$'s conjugate and $\delta_{k\ell}$ is denoted as Kronecker's symbol (whose value is $1$ if $k=\ell$, $0$
  otherwise).

If the support ${\mathcal S}$ of a probability measure over
$\mathbb{R}$ is the finite union of closed compact intervals
${\mathcal S}_k$ for $1\le k\le L$, we will refer to each compact
interval ${\mathcal S}_k$ as a {\em cluster} of ${\mathcal S}$.

If ${\bf Z} \in \mathbb{C}^{N\times N}$ is a nonnegative Hermitian matrix
with eigenvalues $(\xi_i; \ 1\le i\le N)$, we denote in the sequel by
$\mathrm{eig}({\bf Z})=\{\xi_i, 1 \leq i \leq N\}$ the set of its eigenvalues and by $F^{\bf Z}$
the empirical distribution of its eigenvalues (also called
{\em spectral distribution} of $\bf Z$), {\em i.e.}:
$$
F^{\bf Z}(d\,\lambda) = \frac 1N \sum_{i=1}^N \delta_{\xi_i}(d\,\lambda)\ ,
$$
where $\delta_x$ stands for the Dirac probability measure at $x$.

Convergence in distribution will be denoted by
$\xrightarrow[]{\mathcal D}$, in probability by
$\xrightarrow[]{\mathcal P}$; and almost sure convergence, by
$\xrightarrow[]{a.s.}$.

\subsection{Matrix Model}

Consider a $N\times M$ matrix ${\bf X}_N=(X_{ij})$ whose entries are
independent and identically distributed (i.i.d.) random variables,
with distribution ${\mathcal C}{\mathcal N}(0,1)$, i.e.  $X_{ij} = U +
\mathbf{i} V$, where $U,V$ are both i.i.d. real Gaussian random variables
${\mathcal N}(0,\frac{1}{2})$. Let ${\bf R}_N$ be a $N\times N$ Hermitian
matrix with $L$ ($L$ being fixed) distinct eigenvalues $\rho_1<\cdots<
\rho_L$ with respective multiplicities $N_1,\cdots, N_L$ (notice that
$\sum_{i=1}^L N_i = N$). Consider now
$$
{\bf Y}_N = {\bf R}_N^{1/2} {\bf X}_N\ .
$$
The matrix ${\bf Y}_N=[{\bf y}_1,\cdots,{\bf y}_M]$ is the concatenation of
$M$ independent observations $[{\bf y}_1,\cdots,{\bf y}_M]$, where each observation
writes ${\bf y}_i = {\bf R}_N^{1/2}{\bf x}_i$ with ${\bf X}_N= [{\bf
  x}_1,\cdots, {\bf x}_M]$.  In particular, the (population) covariance matrix of each observation ${\bf y}_i$
is ${\bf R}_N =\mathbb{E} {\bf y}_i {\bf y}_i^H$. In this article, we are interested in recovering information on ${\bf R}_N$ based on the
observation
$$
\hat{\bf R}_N = \frac 1M {\bf R}_N^{1/2} {\bf X}_N {\bf X}_N^H {\bf R}_N^{1/2}\ ,
$$
which is referred to as the {\em sample covariance matrix}.

It is in general a complicated task to infer the spectral
properties of ${\bf R}_N$ based on $\hat{\bf R}_N$ for all finite $N,M$. Instead, in
the following, we assume that $N$ and $M$ are large, and consider the
following asymptotic regime:

{\bf{Assumption 1}} (A1): \begin{equation} N,M\to \infty\ ,\quad   \textrm{with}\quad  \frac NM \to c\in (0,\infty)\ ,\quad
\textrm{and}\quad \frac {N_i}M \to c_i\in (0,\infty)\ , \ 1\le i\le L. \end{equation}
This assumption will be shortly referred to as $N,M\to \infty$.

{\bf{Assumption 2}} (A2):

  We assume that the limiting support $\mathcal S$ of the eigenvalue distribution of
  $\hat{\bf R}_N$ is formed of $L$ compact disjoint subsets (cf. Figure \ref{fig:clusters}). Following \cite{MES08b}, one can also reformulate this condition in a more analytic manner: The limiting support of $\hat{\bf R}_N$ is formed of $L$ clusters if and only if
for $i\in \{1,..,L\}$, $\inf_N \{ \frac{M}{N}-\Psi_N(i)\} >0$,
where

$$\Psi_N(i)=
\begin{cases} \frac{1}{N}\sum_{r=1}^L N_r\left( \frac{\rho_r}{\rho_r-\alpha_1}\right)^2 & \text{$m=1$,} \\
\max \Big{\{} \frac{1}{N}\sum_{r=1}^L N_r\left( \frac{\rho_r}{\rho_r-\alpha_{m-1}}\right)^2 ,
\frac{1}{N}\sum_{r=1}^L N_r\left( \frac{\rho_r}{\rho_r-\alpha_m}\right)^2  \Big{\}} & \text{$1 < m < L$,} \\
\frac{1}{N}\sum_{r=1}^L N_r\left( \frac{\rho_r}{\rho_r-\alpha_{L-1}} \right)^2 & \text{$m=L$}
\end{cases}$$
where $\alpha_1 \leq \cdots \leq \alpha_{L-1}$ are $L-1$ different ordered solutions to the equation

$$\frac{1}{N} \sum_{r=1}^{L} N_r \frac{\rho_r^2}{(\rho_r- x)^3}=0.$$
This condition is also called the {\em separability} condition.

Figure \ref{fig:clusters} depicts the eigenvalues of a realization of the random matrix $\hat{\bf R}_N$ and
the associated limiting distribution as $N,M$ grow large, for $\rho_1=1$,
$\rho_2=3$, $\rho_3=10$ and $N=60$ with equal multiplicity.

\subsection{Mestre's Estimator of the population eigenvalues}

In \cite{MES08b}, an estimator of the population covariance matrix eigenvalues $(\rho_k; \
1\le k\le L)$ based on the observations $\hat{\bf R}_N$ is proposed.

\begin{theorem}\cite{MES08b} \label{th:mestre}
  \label{prop:mestre} Denote by
  $\hat{\lambda}_1 \leq \cdots \leq \hat{\lambda}_N$ the ordered eigenvalues of $\hat{\bf R}_N$.
  % and by $\boldsymbol{\lambda} = (\lambda_1,\ldots,\lambda_N)^T$.
  % We further
  % assume that $F^{ {\bf T}_N} \rightarrow F^T$ and $N/M \to c$ for all
  % large $M$ considered.
  Let $M,N\to \infty$ in the sense of the assumption (A1). Under the assumptions (A1)-(A2), the following convergence holds true:
  \begin{equation} \label{eq:estimator}
    \hat{\rho}_k - \rho_k \xrightarrow[M,N\to \infty]{a.s.} 0\ ,
  \end{equation}
  where
  \begin{equation}
          \label{eq:Mestre_tk}
	  \hat{\rho}_k = \frac{M}{N_k}\sum_{m\in \mathcal N_k}\left(\hat\lambda_m - \hat{\mu}_m\right)\ ,
  \end{equation}
  with $\mathcal N_k =\{ \sum_{j=1}^{k-1}N_j +1,\ldots,\sum_{j=1}^kN_j
  \}$ and $\hat\mu_1 \leq \cdots\leq \hat\mu_N$ the (real and) ordered
  solutions of:
% \footnote{Another characterization of interest of the
%     $\hat\mu_i$'s is the fact that they are the eigenvalues of ${\rm
%     diag}(\hat{\bm\lambda})-\frac1M\sqrt{\hat\bm\lambda} \sqrt{\hat\bm\lambda}^T.$, where
%   $\hat\boldsymbol{\lambda} = (\hat\lambda_1,\ldots,\hat\lambda_N)^T$.}
% of
\begin{equation}\label{eq:def-mu}
\frac 1N \sum_{m=1}^N \frac{\hat \lambda_m}{\hat\lambda_m - \mu} = \frac MN\ .
\end{equation}
% being the
%   ordered eigenvalues of ${\rm
%     diag}({\bm\lambda})-\frac1M\sqrt{\bm\lambda} \sqrt{\bm\lambda}^T.$
\end{theorem}

\begin{figure}
  \centering
  \begin{tikzpicture}[font=\footnotesize,scale=0.8]
    \renewcommand{\axisdefaulttryminticks}{8}
    \tikzstyle{every major grid}+=[style=densely dashed]
%    \tikzstyle{every pin}=[fill=white,draw=black,font=\footnotesize,edge style={<-}]
\pgfplotsset{every axis y label/.append style={yshift=-20pt}}
\pgfplotsset{every axis x label/.append style={yshift=5pt}}
    %\tikzstyle{every axis x label}+=[yshift=5pt]
    \tikzstyle{every axis legend}+=[cells={anchor=west},fill=white,%draw=none,
        at={(0.98,0.98)}, anchor=north east, font=\scriptsize ]

    \begin{axis}[
      xmajorgrids=true,
      xlabel={Eigenvalues},
      ylabel={Density},
      xtick = {1,3,10},
      ytick = {},
      yticklabels = {},
      xmin=0,
      xmax=15,
      ymin=0,
      ymax=0.07,
      ]

      \addplot[smooth,red,line width=0.5pt] plot coordinates{
(0,0) (0.301000,-0.000000) (0.401000,-0.000000) (0.501000,-0.000000) (0.601000,-0.000000) (0.701000,0.019786) (0.801000,0.056237) (0.901000,0.061934) (1.001000,0.059629) (1.101000,0.053338) (1.201000,0.044219) (1.301000,0.031803) (1.401000,0.008620) (1.501000,-0.000000) (1.601000,-0.000000) (1.701000,-0.000000) (1.801000,0.000000) (1.901000,0.000000) (2.001000,0.000000) (2.101000,0.010929) (2.201000,0.015193) (2.301000,0.017541) (2.401000,0.018929) (2.501000,0.019721) (2.601000,0.020103) (2.701000,0.020184) (2.801000,0.020037) (2.901000,0.019709) (3.001000,0.019234) (3.101000,0.018636) (3.201000,0.017929) (3.301000,0.017125) (3.401000,0.016229) (3.501000,0.015242) (3.601000,0.014157) (3.701000,0.012965) (3.801000,0.011643) (3.901000,0.010150) (4.001000,0.008405) (4.101000,0.006210) (4.201000,0.002623) (4.301000,-0.000000) (4.401000,-0.000000) (4.501000,-0.000000) (4.601000,-0.000000) (4.701000,-0.000000) (4.801000,-0.000000) (4.901000,-0.000000) (5.001000,-0.000000) (5.101000,-0.000000) (5.201000,-0.000000) (5.301000,0.000000) (5.401000,0.000000) (5.501000,0.000000) (5.601000,0.000000) (5.701000,0.000000) (5.801000,0.000000) (5.901000,0.000000) (6.001000,0.000000) (6.101000,0.000000) (6.201000,0.000000) (6.301000,0.000000) (6.401000,0.000000) (6.501000,0.000000) (6.601000,0.000000) (6.701000,0.000000) (6.801000,0.000000) (6.901000,0.000939) (7.001000,0.002160) (7.101000,0.002848) (7.201000,0.003352) (7.301000,0.003751) (7.401000,0.004078) (7.501000,0.004353) (7.601000,0.004586) (7.701000,0.004786) (7.801000,0.004958) (7.901000,0.005106) (8.001000,0.005234) (8.101000,0.005344) (8.201000,0.005438) (8.301000,0.005519) (8.401000,0.005587) (8.501000,0.005644) (8.601000,0.005691) (8.701000,0.005729) (8.801000,0.005758) (8.901000,0.005780) (9.001000,0.005794) (9.101000,0.005802) (9.201000,0.005804) (9.301000,0.005801) (9.401000,0.005792) (9.501000,0.005779) (9.601000,0.005760) (9.701000,0.005738) (9.801000,0.005711) (9.901000,0.005681) (10.001000,0.005647) (10.101000,0.005610) (10.201000,0.005570) (10.301000,0.005526) (10.401000,0.005479) (10.501000,0.005430) (10.601000,0.005378) (10.701000,0.005323) (10.801000,0.005266) (10.901000,0.005207) (11.001000,0.005145) (11.101000,0.005081) (11.201000,0.005014) (11.301000,0.004945) (11.401000,0.004874) (11.501000,0.004801) (11.601000,0.004726) (11.701000,0.004648) (11.801000,0.004569) (11.901000,0.004487) (12.001000,0.004403) (12.101000,0.004316) (12.201000,0.004228) (12.301000,0.004136) (12.401000,0.004043) (12.501000,0.003947) (12.601000,0.003848) (12.701000,0.003746) (12.801000,0.003642) (12.901000,0.003534) (13.001000,0.003423) (13.101000,0.003307) (13.201000,0.003188) (13.301000,0.003065) (13.401000,0.002936) (13.501000,0.002802) (13.601000,0.002661) (13.701000,0.002513) (13.801000,0.002356) (13.901000,0.002188) (14.001000,0.002007) (14.101000,0.001809) (14.201000,0.001587) (14.301000,0.001330) (14.401000,0.001011) (14.501000,0.000528) (14.601000,-0.000000) (14.701000,-0.000000) (14.801000,-0.000000) (14.901000,-0.000000)
      };
      \addplot[only marks,mark=x,line width=0.5pt] plot coordinates{
(0.725210,0)(0.761293,0)(0.773428,0)(0.804447,0)(0.850258,0)(0.876619,0)(0.914223,0)(0.936101,0)(0.966339,0)(0.983292,0)(1.000177,0)(1.037538,0)(1.049136,0)(1.078080,0)(1.126212,0)(1.159059,0)(1.193031,0)(1.232553,0)(1.295855,0)(1.357692,0)(2.087068,0)(2.322090,0)(2.404135,0)(2.428502,0)(2.523406,0)(2.591069,0)(2.611899,0)(2.760563,0)(2.842153,0)(2.931724,0)(3.005278,0)(3.114150,0)(3.205625,0)(3.330611,0)(3.448666,0)(3.493621,0)(3.634134,0)(3.752682,0)(3.923243,0)(4.018420,0)(7.376547,0)(7.714738,0)(7.892245,0)(8.144817,0)(8.710341,0)(8.864955,0)(9.216528,0)(9.416072,0)(9.976535,0)(10.291326,0)(10.398824,0)(10.794176,0)(11.077766,0)(11.360634,0)(11.786087,0)(12.331401,0)(12.456995,0)(12.879208,0)(13.498515,0)(13.942644,0)
      };
      \legend{{Asymptotic spectrum},{Empirical eigenvalues}}
    \end{axis}
  \end{tikzpicture}
  \caption{Empirical and asymptotic eigenvalue distribution of $\hat{{\bf R}}_N$ for $L=3$, $\rho_1=1$, $\rho_2=3$, $\rho_3=10$, $N/M=c=0.1$, $N=60$, $N_1=N_2=N_3=20$.}
  \label{fig:clusters}
\end{figure}
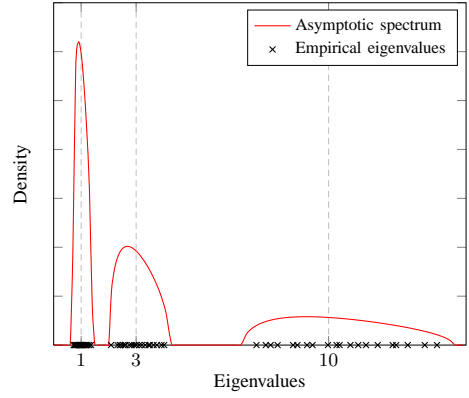

\vspace{4mm}

\subsection{Integral representation of estimator $\hat \rho_k$ - Stieltjes transforms}

The proof of Theorem \ref{prop:mestre} relies on random matrix theory, and in particular,
\cite{MAR67,SIL95} use as a key ingredient the {\it Stieltjes transform}.

The {\it Stieltjes transform} $m_{\mathbb{P}}$ of a probability
distribution $\mathbb{P}$ over $\mathbb{R}^+$ is a $\mathbb{C}$-valued
function defined by:
$$
m_{\mathbb{P}}(z)= \int_{\mathbb{R}^+} \frac{\mathbb{P}(d\lambda)}{\lambda-z}\ , \quad z\in \mathbb{C} \backslash \mathbb{R}^+\ .
$$
There also exists an inverse formula to recover the probability
distribution associated to a Stieljes transform: Let $a < b$ be two continuity points of the cumulative distribution function associated to $\mathbb{P}$, then

$$\mathbb{P}([a,b]) = \frac{1}{\pi}  \lim_{y \downarrow 0} \Im \left[ \int_a^b   m_{\mathbb{P}} (x+iy)dx \right].$$

In the case where $F^{\bf Z}$ is the spectral distribution associated
to a nonnegative Hermitian matrix ${\bf Z}\in\mathbb{C}^{N\times N}$
with eigenvalues $(\xi_i; \ 1\le i\le N)$, the Stieltjes transform
$m_{\bf Z}$ of $F^{\bf Z}$ takes the particular form:
\begin{eqnarray*}
  m_{\bf Z}(z) &=&\int \frac{F^{\bf Z}(d\,\lambda)} {\lambda-z}  \\
&=& \frac1N\sum_{i=1}^N \frac1{\xi_i-z} \ = \ \frac 1N \tr \left( {\bf Z} -z{\bf I}_N\right)^{-1}\ ,
\end{eqnarray*}
which can be seen as the normalized trace of the resolvent $\left( {\bf Z} -z{\bf I}_N\right)^{-1}$.
Since the seminal  paper of Mar\v{c}enko and Pastur \cite{MAR67}, the
Stieltjes transform has proved to be extremely efficient to describe
the limiting spectrum of large dimensional random matrices.

In the following, we recall some elements of the proof of Theorem \ref{th:mestre}, necessary for the
remainder of the article. The first important result is due to Bai and
Silverstein \cite{SIL95} (see also \cite{MAR67}).

\begin{theorem} \cite{SIL95} Denote by
$F^{\bf R}$ the limiting spectral distribution of ${\bf R}_N$, {\em i.e.}
$F^{\bf R}(d\,\lambda) = \sum_{k=1}^L \frac{c_k}{c} \delta_{\rho_k}(d\,\lambda)$. Under the assumption (A1), the spectral distribution $F^{\hat{\bf
    R}_N}$ of the sample covariance matrix $\hat{\bf R}_N$ converges
(weakly and almost surely) to a probability distribution $F$ as
$M,N\to \infty$, whose Stieltjes transform $m(z)$ satisfies:
$$
m(z)=\frac1c \underline{m}(z) -
\left(1-\frac1c\right)\frac1z\ ,
$$
for $z\in \mathbb{C}^+ = \{ z\in \mathbb{C},\ \Im(z)>0\}$, where $\underline{m}(z)$ is
defined as the unique solution in $\mathbb{C}^+$ of:
$$
\underline{m}(z) = - \left( z - c\int \frac{t}{1+t\underline{m}(z)}dF^{\bf R}(t) \right)^{-1}.
$$

\end{theorem}
Note that $\underline{m}(z)$ is also a Stieltjes transform whose
associated distribution function will be denoted $\underline{F}$,
which turns out to be the limiting spectral distribution of
$F^{\hat{\underline{\bf R}}_N}$ where $\hat{\underline{\bf R}}_N$ is
defined as:
$$
\hat{\underline{\bf R}}_N \triangleq \frac{1}{M}{\bf X}_N^H {\bf R}_N{\bf X}_N\ .
$$
Denote by $m_{ \hat{\bf R}_N}(z)$ and $m_{ \hat{\underline{\bf
      R}}_N}(z)$ the Stieltjes transforms of $F^{ \hat{\bf R}_N}$ and
$F^{\hat{\underline{\bf R}}_N}$. Notice in particular that
$$
 m_{ \hat{\bf R}_N}(z) = \frac{M}N m_{ \hat{\underline{\bf
      R}}_N}(z) - \left(1 - \frac{M}N\right)\frac1z\ .
$$
\begin{remark} This relation associated to \eqref{eq:def-mu} readily implies that
$m_{ \hat{\underline{\bf R}}_N}(\hat \mu_i) = 0$. Otherwise stated,
the $\hat\mu_i$'s are the zeros of $m_{ \hat{\underline{\bf
      R}}_N}$. This fact will be of importance in the sequel.
\end{remark}

Denote by $m_N(z)$ and $\underline{m}_N(z)$
the finite-dimensional counterparts of $m(z)$ and $\underline{m}(z)$,
respectively, defined by the relations:
\begin{eqnarray*}
  \underline{m}_N(z) &=& - \left( z - \frac{N}M\int \frac{t}{1+t\underline{m}_N(z)}dF^{{\bf R}_N}(t) \right)^{-1}\ , \\
  m_N(z) &=& \frac{M}N \underline{m}_N(z) - \left(1 - \frac{M}N\right)\frac1z\ .
\end{eqnarray*}
It can be shown that $m_N$ and $\underline{m}_N$ are Stieltjes transforms of given probability measures
$F_N$ and $\underline{F}_N$, respectively (cf. \cite[Theorem 3.2]{COUbook}).

With these notations at hand, we can now derive Theorem \ref{th:mestre}.
By Cauchy's formula, write:
$$
\rho_k = \frac{N}{N_k} \frac{1}{2 i \pi} \oint_{\Gamma_k} \left(
  \frac{1}{N}\sum_{r=1}^{L} N_r \frac{w}{\rho_r-w} dw \right)\ ,
$$
where $\Gamma_k$ is a negatively oriented contour taking values on
$\mathbb{C} \setminus\{\rho_1,\cdots, \rho_L\}$ and only enclosing
$\rho_k$. With the change of variable $w=-\frac{1}{\underline{m}_{M}(z)}$
and the condition that the limiting support $\mathcal S$ of the
eigenvalue distribution of ${\bf R}_N$ is formed of $L$ distinct clusters
$({\mathcal S}_k, 1\le k\le L)$ (cf. Figure \ref{fig:clusters}), we can write:
\begin{equation}
\label{eq:eigenvalue}
  \rho_k = \frac{M}{2 i \pi N_k}\oint_{\mathcal C_k} z\frac{\underline{m}_N'(z)}
{\underline{m}_N(z)}dz\ ,\quad 1\le k\le L\
\end{equation}
where ${\mathcal C}_k$ and ${\mathcal C}_\ell$ denote negatively
oriented contours which enclose the corresponding clusters ${\mathcal S_k}$ and
${\mathcal S_\ell}$ respectively.
% The proof of the proposition \ref{prop:1} is based on the link between
% m(z) and T. At first , by Cauchy formula \cite{Mars87}, Mestre
% writes $$t_k= \frac{N}{N_k} \frac{1}{2 i \pi} \oint_{\mathcal C_k}
% \left( \frac{1}{N}\sum_{r=1}^{L} N_r \frac{w}{t_r-w} dw \right)$$ with
% $\mathcal C_k$ a negatively oriented contour that circles around the
% $k$-th cluster in $\mathcal S$ only. Then with the change of variables
% by $w=-\frac{1}{\underline{m}_{M}(z)}$ and by the condition that the
% limiting support $\mathcal S$ of the eigenvalue distribution of ${\bf
%   R}_N$ is formed of $L$ compact disjoint subsets (Figure
% \ref{fig:clusters}), finally Mestre writes $t_k$ explicitly as the
% following complex integral of $m(z)$ \cite[Chapter 6]{COUbook}
% \begin{equation*}
%   t_k = \frac{M}{2 i \pi N_k}\oint_{\mathcal C_k} z\frac{\underline{m}_N'(z)}{\underline{m}_N(z)}dz.
% \end{equation*}
Defining
\begin{equation}
	\label{eq:hattk}
        \hat{\rho}_k \triangleq \frac{M}{2\pi i N_k}\oint_{\mathcal C_k} z\frac{m_{\hat{\underline{\bf R}}_N}'(z)}{m_{\hat{\underline{\bf R}}_N}(z)}dz\ ,\quad 1\le k\le L\ ,
\end{equation}
dominated convergence arguments ensure that $\rho_k-\hat{\rho}_k\to
0$, almost surely. The integral form of $\hat{\rho}_k$ can then be
explicitly computed thanks to residue calculus, and this finally
yields \eqref{eq:Mestre_tk}.

The main objective of this article is to study the performance of
the estimators $(\hat{\rho}_k,\ 1\le k\le L)$. More precisely, we will establish a
central limit theorem (CLT) for $(M(\hat{\rho}_k- \rho_k),\ 1\le k\le L)$ as
$M,N\to \infty$, explicitly characterize the limiting covariance
matrix $\boldsymbol{\Theta}= (\Theta_{k\ell})_{1\le k,\ell\le L}$, and
finally provide an estimator for $\boldsymbol{\Theta}$.

\section{Fluctuations of the population eigenvalue estimators}
\label{sec:CLT}

\subsection{The Central Limit Theorem}

The main result of this article is the following CLT which expresses the fluctuations of $(\hat{\rho}_k,\ 1\le k\le L)$.

\begin{theorem}\label{th:CLT}
Under the assumptions (A1)-(A2) and with the same notations:
$$
\left( M(\hat{\rho}_k- \rho_k), \ 1\le k\le L \right)
 \xrightarrow[M,N\to \infty]{\mathcal D} {\bf x} \sim {\mathcal N}_L(0,\boldsymbol{\Theta})\ ,
$$
where ${\mathcal N}_L$ refers to a real $L$-dimensional
Gaussian distribution, and $\boldsymbol{\Theta}$ is a $L\times L$
matrix whose entries $\Theta_{k\ell}$ are given by \eqref{eq:def-theta},
where ${\mathcal C}_k$ and ${\mathcal C}_\ell$ are defined as before (cf. Formula (\ref{eq:eigenvalue})).
\end{theorem}

\begin{figure*}[t!]
  \begin{eqnarray}
    \label{eq:def-theta}
    \Theta_{k\ell} \quad = \quad -\frac1{4\pi^2 c_{k}c_{\ell}}\oint_{\mathcal C_{k}}\oint_{\mathcal C_{\ell}}
\left[ \frac{\underline{m}'(z_1)\underline{m}'(z_2)}{(\underline{m}(z_1)-\underline{m}(z_2))^2}
-\frac1{(z_1-z_2)^2} \right] \frac1{\underline{m}(z_1)\underline{m}(z_2)}dz_1 dz_2\ .
  \end{eqnarray}
    \hrulefill
\end{figure*}

\subsection{Proof of Theorem \ref{th:CLT}}
We first outline the main steps of the proof and then provide the details.

Using the integral representation of $\hat{\rho}_k$ and $\rho_k$, we get: Almost surely,
\begin{eqnarray*}
  M(\hat{\rho}_k - \rho_k) &=& \frac{M^2}{2\pi i N_k}\oint_{\mathcal C_k} z \left( \frac{m_{\hat{\underline{\bf R}}_N}'(z)}{m_{\hat{\underline{\bf R}}_N}(z)}  - \frac{\underline{m}_N'(z)}
    {\underline{m}_N(z)}\right) \,dz
\end{eqnarray*}
Denote by $C({\mathcal C}_k,\mathbb{C})$ the set of continuous
functions from ${\mathcal C}_k$ to $\mathbb{C}$ endowed with the
supremum norm $\|u\|_{\infty} = \sup_{{\mathcal C}_k}|u|$. Consider
the process:
$$
(X_N,X'_N,u_N,u'_N):{\mathcal C}_k \to \mathbb{C}^4
$$
where
\begin{eqnarray*}
X_N(z)&=& M\left(m_{\hat{\underline{\bf R}}_N}(z) -
\underline{m}_N(z)\right)\ ,\\
X'_N(z)&=& M\left(m_{\hat{\underline{\bf R}}_N}'(z) -
\underline{m}_N'(z)\right)\ ,\\
u_N(z)&=&m_{\hat{\underline{\bf R}}_N}(z)\ ,\quad u_N'(z) \ =\ m_{\hat{\underline{\bf R}}_N}'(z)\ .
\end{eqnarray*}
Then due to 'no eigenvalue' result (cf. \cite{SIL98}, see also Proposition \ref{prop:eigen-outside-bulk}), $(X_N,X'_N,u_N,u'_N)$ almost surely belongs to $C({\mathcal C}_k,\mathbb{C})$ and
$M(\hat{\rho}_k - \rho_k)$ writes:
\begin{eqnarray*}
M(\hat{\rho}_k - \rho_k) & =& \frac{M}{2\pi i N_k}\oint_{\mathcal C_k}
z \left( \frac{\underline{m}_N(z) X'_N(z) - u'_N(z) X_N(z)}{\underline{m}_N(z) u_N(z)}\right) \, dz\\
&\stackrel{\triangle}=& \Upsilon_N(X_N,X'_N,u_N,u'_N)\ ,
\end{eqnarray*}
where
\begin{equation}\label{eq:def-upsilon}
\Upsilon_N(x,x',u,u')\ =\ \frac{M}{2\pi i N_k}\oint_{\mathcal C_k}
z \left( \frac{\underline{m}_N(z) x'(z) - u'(z) x(z)}{\underline{m}_N(z) u(z)}\right) \, dz\ .
\end{equation}
If needed, we shall explicitly indicate the dependence in the contour ${\mathcal C}_k$ and write
$\Upsilon_N(x,x',u,u',{\mathcal C}_k)$.
The main idea of the proof of the theorem lies in three steps:
\begin{enumerate}
\item[(i)] To
prove the convergence in distribution of the process
$(X_N,X'_N,u_N,u'_N)$ to a Gaussian process.
\item[(ii)] To transfer this
convergence to the quantity $\Upsilon_N(X_N,X'_N,u_N,u'_N)$ with the
help of the continuous mapping theorem \cite{KAL02}.
\item[(iii)] To check that the limit (in distribution) of
  $\Upsilon_N(X_N,X'_N,u_N,u'_N)$ is Gaussian and to compute the
  limiting covariance between $\Upsilon_N(X_N,X'_N,u_N,u'_N,{\mathcal C}_k)$ and
  $\Upsilon_N(X_N,X'_N,u_N,u'_N,{\mathcal C}_\ell)$.
\end{enumerate}

\begin{remark}
Note that the convergence in step (i) is a distribution convergence
at a process level, hence one has to first establish the finite
dimensional convergence of the process and then to prove that the
process is tight over ${\mathcal C}_k$. Tightness turns out to be
difficult to establish due to the lack of control over the eigenvalues
of $\hat{\underline{\bf R}}_N$ whenever the contour crosses the real line.
In order to circumvent this issue, we shall introduce, following Bai
and Silverstein \cite{BAI04}, a process that approximates $X_N$ and $X'_N$.
\end{remark}

Let us now start the proof of Theorem \ref{th:CLT}.

\begin{lemma}\label{lemma:convergence-approx} Under the assumptions (A1)-(A2), the process
$$
(X_N,X'_N):{\mathcal C}_k \to \mathbb{C}^4
$$
converges in distribution to a Gaussian process $(X,Y)$
with mean function zero and covariance function:
\begin{eqnarray}
  \cov(X(z), X(\tilde z))&=& \frac{\underline{m}'(z) \underline{m}'(\tilde z)}{(\underline{m}(z) - \underline{m}(\tilde z) )^2} - \frac 1{(z -\tilde z)^2}\ \stackrel{\triangle}= \ \kappa(z,\tilde z)\ ,\label{eq:def-kappa}\\
 \cov(Y(z), X(\tilde z))&=& \frac{\partial}{\partial z} \kappa(z,\tilde z)\ ,\nonumber \\
\cov(X(z), Y(\tilde z))& =&  \frac{\partial}{\partial \tilde z} \kappa(z,\tilde z)\ , \nonumber \\
\cov(Y(z), Y(\tilde z)) & =& \frac{\partial^2}{\partial z\partial \tilde z} \kappa(z,\tilde z) \nonumber\ .
\end{eqnarray}
\end{lemma}
Lemma \ref{lemma:convergence-approx} is the cornerstone to the proof
of Theorem \ref{th:CLT}. 	The proof of Lemma \ref{lemma:convergence-approx}
is postponed to Appendix \ref{app:convergence-approx} and relies on
the following proposition, of independent interest:

\begin{proposition}\label{prop:eigen-outside-bulk} Under the assumptions (A1)-(A2)
  and denote by ${\mathcal S}$ the support of the
  probability distribution associated to the Stieltjes transform
  $m$. Then, for every $\varepsilon>0$, $\ell \in \mathbb{N}^*$:
$$
  \mathbb{P}\left( \sup_{\lambda \in \mathrm{eig}({\bf \hat R_N})}
    d(\lambda,{\mathcal S}) >\varepsilon \right) = {\mathcal O}\left(
    \frac 1{N^{\ell}}\right) \ ,
$$
where $d(\lambda,{\mathcal S}) = \inf_{x\in {\mathcal S}} |\lambda - x|$.
\end{proposition}
The proof of Proposition \ref{prop:eigen-outside-bulk} is postponed to
Appendix \ref{app:eigen-outside-bulk}.

As $(u_N,u_N') \xrightarrow[N,M\to \infty]{a.s.}
(\underline{m},\underline{m}')$, a straightforward corollary of Lemma
\ref{lemma:convergence-approx} yields the convergence in distribution
of $(X_N,X'_N,u_N,u_N')$ to $(X,Y,\underline{m},\underline{m}')$. This
concludes the proof of step (i).

A direct consequence of Lemma \ref{lemma:convergence-approx}
yields that $ (X_N,X'_N,u_N,u'_N):{\mathcal C}_k \to \mathbb{C}^4 $
converges in distribution to the Gaussian process $(X,Y,\underline{m},\underline{m}')$
defined as before.  We are now in
position to transfer the convergence of $(X_N,X'_N,u_N,u'_N)$ to
$\Upsilon_N(X_N,X'_N,u_N,u'_N)$ via the continuous mapping theorem,
whose statement as expressed in \cite{KAL02} is reminded below.
\begin{proposition}[cf. {\cite[Th. 4.27]{KAL02}}]
  \label{th:CMT}
  For any metric spaces $S_1$ and $S_2$, let $\xi$, $ (\xi_n)_{n\ge
  1}$ be random elements in $S_1$ with $\xi_n \xrightarrow[n\to
  \infty]{\mathcal D} \xi$ and consider some measurable mappings
  $f$, $(f_n)_{n\ge 1}$: $S_1 \mapsto S_2$ and a measurable set $\Gamma \subset S_1$
  with $\xi \in \Gamma$ a.s. such that $f_n(s_n) \rightarrow f(s)$ as $s_n
  \rightarrow s \in \Gamma$. Then $f_n(\xi_n) \xrightarrow[n\to
  \infty]{\mathcal D} f(\xi)$.
\end{proposition}

It remains to apply Theorem \ref{th:CMT} to the process
$(X_N,X'_N,u_N,u'_N)$ and to the function $\Upsilon_N$ as defined in
\eqref{eq:def-upsilon}. Denote by\footnote{As previously, we shall
  explicitly indicate the dependence on the contour ${\mathcal C}_k$
  if needed and write $\Upsilon(x,x',u,u',{\mathcal C}_k)$.}
$$
\Upsilon(x,y,v,w)\ =\ \frac{1}{2\pi i c_k}\oint_{\mathcal C_k}
z \left( \frac{\underline{m}(z) y(z) - w(z) x(z)}{\underline{m}(z) v(z)}\right) \, dz\ ,
$$
and consider the set
$$
\Gamma=\left\{ (x,y,v,w) \in C^4({\mathcal C}_k,\mathbb{C})\ ,\ \inf_{{\mathcal C}_k}|v|>0 \right\}\ .
$$
Then, it is shown in \cite[Section 9.12.1]{SIL06} that $\inf_{{\mathcal C}_k}|\underline{m}|>0 $,
and, by a dominated convergence theorem argument, that
$(x_N,y_N,v_N,w_N)\to (x,y,v,w) \in \Gamma$ implies that
$\Upsilon_N(x_N,y_N,v_N,w_N)\to \Upsilon(x,y,v,w)$. Therefore,
Theorem \ref{th:CMT} applies to $\Upsilon_N(x_N,y_N,v_N,w_N)$ and
the following convergence holds true:
$$
\Upsilon_N(X_N,X_N',u_N,u_N') \xrightarrow[M,N\to \infty]{\mathcal D}
\Upsilon(X,Y,\underline{m},\underline{m}')\ ,
$$
and step (ii) is established.

It now remains to prove step (iii), {\em i.e.} to check the
Gaussianity of the random variable
$\Upsilon(X,Y,\underline{m},\underline{m}')$ and to compute the
covariance between
$\Upsilon(X,Y,\underline{m},\underline{m}',{\mathcal C}_k)$ and
$\Upsilon(X,Y,\underline{m},\underline{m}',{\mathcal C}_\ell)$.

In order to propagate the Gaussianity of the deviations in the
integrands of \eqref{eq:hattk} to the deviations of the integral which
defines $\hat{\rho}_k$, it suffices to notice that the integral can be
written as the limit of a finite Riemann sum and that a finite Riemann
sum of Gaussian random variables is still Gaussian. Therefore
$M(\hat{\rho}_k-\rho_k)$ converges to a Gaussian distribution. As
$\inf_{z\in \mathcal C_k} |\underline{m}(z)| >0$, a straightforward
application of Fubini's theorem together with the fact that
$\mathbb{E} (X)=\mathbb{E} (Y)=0$ yields:
$$
\mathbb{E} \oint \left( z \frac{\underline{m}^{'}(z) X(z)}
{\underline{m}^2(z)}- z \frac{Y(z)}{\underline{m}(z)} \right) dz =0\ .
$$
It remains to compute the covariance between
$\Upsilon(X,Y,\underline{m},\underline{m}', {\mathcal C}_k)$ and
$\Upsilon(X,Y,\underline{m},\underline{m}', {\mathcal C}_\ell)$ for
possibly different contours ${\mathcal C}_k$ and ${\mathcal C}_\ell$.
We shall therefore evaluate, for $1\le k,\ell\le L$:
\begin{eqnarray*}
  \Theta_{k\ell} &=& \mathbb{E} \left( \Upsilon(X,Y,\underline{m},\underline{m}', {\mathcal C}_k)
    \Upsilon(X,Y,\underline{m},\underline{m}', {\mathcal C}_\ell)
  \right) \ ,  \\
  &\stackrel{(a)}=& -\frac{1}{4 \pi^2 c_k c_l}\oint_{{\mathcal C}_k} \oint_{{\mathcal C}_\ell} z_1 z_2 \Big{(}
  \frac{\underline{m}^{'}(z_1) \underline{m}^{'}(z_2) \kappa(z_1,z_2)}
  {\underline{m}^2 (z_1) \underline{m}^2(z_2)}- \frac{ \underline{m}^{'}(z_1) \partial_2 \kappa(z_1,z_2)}
  { \underline{m}^2 (z_1) \underline{m}(z_2)} \\ && \quad -\frac{\underline{m}^{'}(z_2) \partial_1 \kappa(z_1,z_2)}{\underline{m}(z_1) \underline{m}^2(z_2)} + \frac{\partial_{12}^2 \kappa(z_1,z_2)}{\underline{m} (z_1) \underline{m}(z_2)}\Big{)} dz_1 dz_2\ ,%\label{eq:variance} \\
\end{eqnarray*}
where $(a)$ follows from the fact that $\inf_{z \in \mathcal C_k}
|\underline{m}(z)| >0$ together with Fubini's theorem, and
$\partial_1, \partial_2, \partial_{12}^2$ respectively stand for
$\partial /\partial z_1$, $\partial /\partial z_2$ and  $\partial^2 /\partial z_1 \partial z_2$.

By integration by parts, we obtain

\begin{eqnarray*}
\lefteqn{\oint \frac{z_1 z_2 \underline{m}'(z_2)
\partial_1 \kappa(z_1,z_2) }{{\underline{m}(z_1)\underline{m}^2(z_2)}} dz_1 } \\
& =& \oint \left( -\frac{z_2 \underline{m}'(z_2)  \kappa(z_1,z_2)}{{\underline{m}(z_1)\underline{m}^2(z_2)}}+ \frac{z_1z_2\underline{m}'(z_1) \underline{m}'(z_2)  \kappa(z_1,z_2)} {\underline{m}^2(z_1)\underline{m}^2(z_2)} \right) dz_1\ .
\end{eqnarray*}
Similarly,
\begin{eqnarray*}
\lefteqn{\oint \frac{z_1 z_2 \underline{m}(z_2)  \partial_{1,2} \kappa(z_1,z_2) }{{\underline{m}(z_1)\underline{m}^2(z_2)}} dz_1 }\\
&=& -\oint \frac{z_2  \partial_2 \kappa(z_1,z_2)}{{\underline{m}(z_1)\underline{m}(z_2)}} dz_1 + \oint \frac{z_1z_2\underline{m}'(z_1)  \partial_2 \kappa(z_1,z_2)} {\underline{m}^2(z_1)\underline{m}(z_2)}dz_1\ .
\end{eqnarray*}
Hence
$$
\Theta_{k\ell} =  -\frac{1}{4 \pi^2 c_kc_l}\left\{ \oint_{{\mathcal C}_k} \oint_{{\mathcal C}_\ell} \frac{z_2 \underline{m}'(z_2)
\kappa(z_1,z_2) } {{\underline{m}(z_1)\underline{m}^2(z_2)}} dz_1 dz_2 - \oint_{{\mathcal C}_k}
\oint_{{\mathcal C}_\ell} \frac{z_2
\partial_2 \kappa(z_1,z_2) }{{\underline{m}(z_1)\underline{m}(z_2)}} dz_1 dz_2\right\}\ .
$$
Another integration by parts yields
$$
\oint \frac{z_2 \partial_2
  \kappa(z_1,z_2)}{{\underline{m}(z_1)\underline{m}(z_2)}} dz_2=
-\oint \frac{\kappa(z_1,z_2)}{{\underline{m}(z_1)\underline{m}(z_2)}}
dz_2 + \oint \frac{z_2 \underline{m}'(z_2)
  \kappa(z_1,z_2)}{{\underline{m}(z_1)\underline{m}^2(z_2)}} dz_2\ .
$$
Finally, we obtain:
$$
\Theta_{k\ell} = -\frac{1}{4 \pi^2 c_kc_l} \oint_{{\mathcal C}_k} \oint_{{\mathcal C}_\ell} \frac{\kappa(z_1,z_2)}
{\underline{m}(z_1)\underline{m}(z_2) } d\, z_1 d\, z_2\ ,
$$
and \eqref{eq:def-theta} is established.

\begin{figure*}[t!]
  \begin{equation}
  \label{eq:est}
  \hat{\Theta}_{k\ell} = -\frac{M^2}{4 \pi^2 N_k N_\ell}
\oint_{{\mathcal C}_k} \oint_{{\mathcal C}_\ell}
  \left(\frac{m_{\underline{\hat{\bf R}}_N}'(z_1) m_{\underline{\hat{\bf R}}_N}'(z_2)} {(m_{\underline{\hat{\bf R}}_N}(z_1)- m_{\underline{\hat{\bf R}}_N}(z_2))^2}-\frac{1}{(z_1-z_2)^2}\right) \times \frac{1}{m_{\underline{\hat{\bf R}}_N}(z_1) m_{\underline{\hat{\bf R}}_N}(z_2)} d\, z_1 d\, z_2\ .
    \end{equation}
    \hrulefill
\end{figure*}

\vspace{4mm}

\section{Estimation of the covariance matrix}\label{sec:estimation-cov}

Theorem \ref{th:CLT} describes the limiting performance of the
estimator of Theorem \ref{th:mestre}, with an exact
characterization of its variance. Unfortunately, the variance
$\boldsymbol{\Theta}$ depends upon unknown quantities. We provide
hereafter consistent estimates $\hat{\boldsymbol{\Theta}}$ for
$\boldsymbol{\Theta}$ based on the observations $\hat{\bf R}_N$.

\begin{theorem}
  \label{th:estimate} Assume that the assumptions (A1)-(A2) hold true, and recall the definition of $\Theta_{k\ell}$ given in
  \eqref{eq:def-theta}. Let $\hat{\Theta}_{k\ell}$ be defined by
  \eqref{eq:def-theta-hat}, where $({\mathcal N}_k)$ and $(\hat{\mu}_k)$ are
  defined in Theorem \ref{th:mestre}, then:
  \begin{equation*}
	  \hat{\Theta}_{k\ell} - \Theta_{k\ell} \xrightarrow[]{a.s.} 0
  \end{equation*}
as $N,M\to \infty$.
\end{theorem}

Theorem \ref{th:estimate} is useful in practice as one can obtain
simultaneously an estimate $\hat{\rho}_k$ of the values of $\rho_k$ as well
as an estimation of the degree of confidence for each $\hat{\rho}_k$.

\begin{figure*}[t!]
\begin{multline}
  \label{eq:def-theta-hat}
    \hat{\Theta}_{k\ell} \quad = \quad \frac{M^2}{
      N_{k}N_{\ell}}\left[ \sum_{(i,j)\in\mathcal N_{k}\times\mathcal
        N_{\ell},\ i \neq j}
      \frac{-1}{(\hat\mu_i-\hat\mu_j)^2 m_{\underline{\hat{\bf R}}_N}'(\hat\mu_i) m_{\underline{\hat{\bf R}}_N}'(\hat\mu_j)}\right. \\
    \left.+ \delta_{k\ell}\sum_{i\in\mathcal N_{k}}\left(
\frac{
        m_{\underline{\hat{\bf R}}_N}'''(\hat\mu_i)}{6
        m_{\underline{\hat{\bf R}}_N}'(\hat\mu_i)^3}-\frac{
        m_{\underline{\hat{\bf R}}_N}'' (\hat\mu_i)^2}{4
        m_{\underline{\hat{\bf R}}_N}'(\hat\mu_i)^4} \right) \right]\ .
  \end{multline}
    \hrulefill
\end{figure*}
\vspace{4mm}

\begin{proof} In view of formula \eqref{eq:def-theta}, and taking into
  account the fact that $m_{\underline{\hat{\bf R}}_N}$ and
  $m_{{\underline{\hat{\bf R}}_N}}'$ are consistent estimates for
  $\underline{m}$ and $\underline{m}'$, it is natural to define
  $\hat{\Theta}_{k\ell}$ by replacing the unknown quantities
  $\underline{m}$ and $\underline{m}'$ in \eqref{eq:def-theta} by
  their empirical counterparts $m_{\underline{\hat{\bf R}}_N}$ and
  $m_{{\underline{\hat{\bf R}}_N}}'$, hence the definition of $\hat{\Theta}_{k\ell}$ in \eqref{eq:est}.

  The proof of Theorem \ref{th:estimate} now breaks down into two
  steps: The convergence of $\hat{\Theta}_{k\ell}$ to
  $\Theta_{k\ell}$, which relies on the definition \eqref{eq:est} of
  $\hat{\Theta}_{k\ell}$ and on a dominated convergence argument, and
  the effective computation of the integral in \eqref{eq:est} which
  relies on Cauchy's residue theorem \cite{Mars87}, and yields \eqref{eq:def-theta-hat}.

  We first address the convergence of $\hat{\Theta}_{k\ell}$ to
  $\Theta_{k\ell}$. Due to \cite{SIL98,BAI99}, almost surely, the
  eigenvalues of ${\underline{\hat{\bf R}}_N}$ will eventually belong
  to any $\varepsilon$-blow-up of the support $\underline{\mathcal S}$
  of the probability measure associated to $\underline {m}$, {\it{i.e.}} the set $\{x \in \mathbb{R}: d(x, \underline{\mathcal S}) < \varepsilon\}$. Hence, if $\varepsilon$ is small enough, the distance between these
  eigenvalues and any $z\in {\mathcal C}_k$ will be eventually
  uniformly lower-bounded. By \cite[Lemma 1]{MES08b}, the same result
  holds true for the zeros of $m_{{\underline{\hat{\bf R}}_N}}$
  (which are real). In particular, this implies that
  $m_{{\underline{\hat{\bf R}}_N}}$ is eventually uniformly
  lower-bounded on ${\mathcal C}_k$ (if not, then by compacity, there
  would exist $z\in {\mathcal C}_k$ such that $m_{{\underline{\hat{\bf
          R}}_N}}(z)=0$ which yields a contradiction because all the
  zeroes of $m_{{\underline{\hat{\bf R}}_N}}$ are strictly within the
  contour). With these arguments at hand, one can easily apply the
  dominated convergence theorem and conclude that a.s.
  $\hat{\Theta}_{k\ell} \to \Theta_{k\ell}$.

  We now evaluate the integral \eqref{eq:est} by computing the
  residues of the integrand within ${\mathcal C}_k$ and ${\mathcal
    C}_\ell$. There are two cases to discuss depending on whether
  $k\neq\ell$ and $k= \ell$. Denote by $h(z_1,z_2)$ the integrand in
  \eqref{eq:est}, that is:

\begin{equation}\label{eq:integrand}
  h(z_1,z_2) = \left(\frac{m_{\underline{\hat{\bf R}}_N}'(z_1) m_{\underline{\hat{\bf R}}_N}'(z_2)} {(m_{\underline{\hat{\bf R}}_N}(z_1)- m_{\underline{\hat{\bf R}}_N}(z_2))^2}-\frac{1}{(z_1-z_2)^2}\right) \times \frac{1}{m_{\underline{\hat{\bf R}}_N}(z_1) m_{\underline{\hat{\bf R}}_N}(z_2)}.
\end{equation}

  We first consider the case where $k\neq \ell$.

% We will try to show the theorem \ref{th:estimate} by residue calculus (\cite{Mars87}) for
%   the double integral. As $m_{\underline{\bf R}_N}$
%   (resp. $m'_{\underline{\bf R}_N}$) converges to $\underline{m}$
%   (resp. to $\underline{m}'$) almost surely for growing $M,N$ and with probability one, the zeros of $m_{\underline{\bf R}_N}$ are in the support $\mathcal S$
%   (see \cite[Lemma 1]{MES08b}), we can show easily by the Dominated
%   Convergence Theorem that formula (\ref{eq:est}) converges to
%   $\Theta_{i,j}$ almost surely. Now we will calculate the formula (\ref{eq:est}) by the residual theorem. There are two cases to discuss
%   depending on whether $i=j$ or $i\neq j$.

  In this case, the two integration contours are different and it can
  be assumed that they never intersect (so it can always be assumed
  that $z_1\neq z_2$). Let $z_2$ be fixed, and denote by $\hat{\mu}_i$
  the zeroes (labeled in increasing order) of $m_{\underline{\hat{\bf
        R}}_N}$, then the computation of the
  residue $\mathrm{Res}(h(\cdot,z_2),\hat{\mu}_i)$ of $h(\cdot,z_2)$
  at a zero $\hat{\mu}_i$ of $m_{\underline{\hat{\bf R}}_N}$ which is
  located within ${\mathcal C}_k$ is straightforward and yields:
\begin{equation}\label{eq:definition-r}
r(z_2)\stackrel{\triangle}= \mathrm{Res}(h(\cdot,z_2),\hat{\mu}_i) =  \left(
  \frac{m_{\underline{\hat{\bf R}}_N}'(\hat\mu_i)\,
    m_{\underline{\hat{\bf R}}_N}'(z_2)}{m_{\underline{\bf{\hat R}}_N}^2(z_2) }
  -\frac{1}{(\hat\mu_i-z_2)^2} \right)\frac{1}{m_{\underline{\hat{\bf
      R}}_N}'(\hat\mu_i)m_{\underline{\hat{\bf R}}_N}(z_2)}\ .
\end{equation}
Similarly, if one computes $\mathrm{Res}(r,\hat{\mu}_j)$ at a zero $\hat{\mu}_j$ of
$m_{\underline{\hat{\bf R}}_N}$ located within ${\mathcal C}_k$, one obtains:
$$
\mathrm{Res}(r,\hat{\mu}_j) =
  -\frac{1}{(\hat\mu_i-\hat\mu_j)^2\,m_{\underline{\hat{\bf
      R}}_N}'(\hat\mu_i)m_{\underline{\hat{\bf R}}_N}'(\hat\mu_j)}\ .
$$

Then we need to consider the residue $\xi$ on the set ${\mathcal R}_{z_2}=\{z_1:
  m_{\underline{\hat{\bf R}}_N}(z_1)=m_{\underline{\hat{\bf
        R}}_N}(z_2)\neq 0,\ z_1 \neq z_2 \}$. (If this set is empty, then the residue is zero.) Notice that $\xi$ is not a residue of $\frac{1}{(z_1-z_2)^2} \frac{1}{m_{\underline{\hat{\bf{R}}}_N}(z_1) m_{\underline{\hat{\bf{R}}}_N}(z_2)}$, hence one needs to compute

$$g(z_1,z_2)=\frac{m'_{\underline{\hat{\bf{R}}}_N}(z_1) m'_{\underline{\hat{\bf{R}}}_N}(z_2)}{(m_{\underline{\hat{\bf{R}}}_N}(z_1)-m_{\underline{\hat{\bf{R}}}_N}(z_1))^2 } \frac{1}{m_{\underline{\hat{\bf{R}}}_N}(z_1) m_{\underline{\hat{\bf{R}}}_N}(z_2)}$$
for the residue $\xi$.
By integration by parts, one gets $$
\oint g(z_1,z_2)dz_1 =-\oint \frac{m'_{\underline{\hat{\bf{R}}}_N}(z_1) m'_{\underline{\hat{\bf{R}}}_N}(z_2)}{(m_{\underline{\hat{\bf{R}}}_N}(z_1)-m_{\underline{\hat{\bf{R}}}_N}(z_2))} \frac{dz_1}{m_{\underline{\hat{\bf{R}}}_N}^2(z_1) m_{\underline{\hat{\bf{R}}}_N}(z_2)}.$$
Let $k =\min\{i\in \mathbb{N}^{*}: m_{\underline{\hat{\bf{R}}}_N}^{(i)}(\xi) \neq 0 \}$, then by a Taylor expansion

$$ m_{\underline{\hat{\bf{R}}}_N}(z_1)=m_{\underline{\hat{\bf{R}}}_N}(z_2)+\frac{(z_1-\xi)^{k}}{k!} m_{\underline{\hat{\bf{R}}}_N}^{(k)}(\xi)+o(z_1-\xi)^{k}, $$ and
$$ m_{\underline{\hat{\bf{R}}}_N}'(z_1) =  \frac{(z_1-\xi)^{k-1}}{(k-1)!} m_{\underline{\hat{\bf{R}}}_N}^{(k)}(\xi)+o(z_1-\xi)^{k-1}.$$
Hence $$\mathrm{Res}(g, \xi) =-\frac{k m'_{\underline{\hat{\bf{R}}}_N}(z_2)}{m_{\underline{\hat{\bf{R}}}_N}^3(z_2)}.$$
As it is the derivative function of $\frac{k}{2 m_{\underline{\hat{\bf{R}}}_N}^2(z_2)}$, the integration with respect to $z_2$ is zero.

It remains to count the number of zeros within each contour. By
\cite[Lemma 1]{MES08b}, eventually, there are exactly as many zeros as
eigenvalues within each contour. It has been proved that the
contribution of the residues of $\xi$ on ${\mathcal R}_{z_2}$ is null,
hence the result in the case $k\neq \ell$:
$$
\hat \Theta_{k\ell} = - \frac{M^2}{ N_k N_\ell}
  \sum_{(i,j)\in\mathcal N_k\times\mathcal N_\ell}-\frac1{( \hat\mu_i-\hat\mu_j)^2
    m_{\underline{\hat{\bf R}}_N}'(\mu_i) m_{\underline{\hat{\bf
        R}}_N}'(\mu_j)}\ .
$$

We now compute the integral \eqref{eq:est} in the case where $k=\ell$,
and begin by the computation of the residues at $\hat\mu_i$. The
definition \eqref{eq:definition-r} of $r$ and the computation of
$\mathrm{Res}(r,\hat{\mu}_j)$ still hold true in the case where
$\hat\mu_j$ is within ${\mathcal C}_k$ but different from
$\hat\mu_i$. It remains to compute
$\mathrm{Res}(r,\hat{\mu}_i)$. Taking $z_2 \to \mu_i$, we get:
\begin{eqnarray*}
\lim_{z_2 \to \hat\mu_i} (z_2-\hat\mu_i)^3 \left(\frac{1}{m'_{\underline{\hat{{\bf R}}}_N}(\hat\mu_i) m_{\underline{\hat{{\bf R}}}_N}(z_2)(\hat\mu_i- z_2)^2} \right) &=& \frac{1}{{m'}_{\underline{\hat{{\bf R}}}_N}^2(\hat\mu_i)},\ \\
\lim_{z_2 \to \hat\mu_i} (z_2-\hat\mu_i)^2 \left( \frac{1}{m'_{\underline{\hat{{\bf R}}}_N}(\hat\mu_i)
m_{\underline{\hat{\bf R}}_N}(z_2)(\hat\mu_i- z_2)^2} - \frac{1}{{m'_{\underline{\hat{\bf R}}_N}}^2(\hat\mu_i) (z_2- \hat\mu_i)^3 } \right) &= & -\frac{m_{\underline{\hat{\bf R}}_N}''(\hat\mu_i)}{2{m'_{\underline{\hat{\bf R}}_N}}^3
(\hat\mu_i)}\ .
\end{eqnarray*}
Finally,
$$\begin{aligned}
\lim_{z_2 \to \hat\mu_i} & (z_2-\hat\mu_i) \left( \frac{1}{m'_{\underline{\hat{\bf R}}_N}(\hat\mu_i)
    m_{\underline{\hat{\bf R}}_N}(z_2)(\hat\mu_i- z_2)^2} - \frac{1}{{m'}_{\underline{\hat{\bf R}}_N}^2(\hat\mu_i) (z_2- \hat\mu_i)^3 }+ \frac{m_{\underline{\bf
        R}_N}''(\hat\mu_i)}{2{m'}_{\underline{\bf
        R}_N}^3(\hat\mu_i)(z_2-\hat\mu_i)^2} \right)\\
&= \frac{{m'''}_{\underline{\hat{\bf R}}_N}(\hat\mu_i)}{6{m'}_{\underline{\hat{\bf R}}_N}(\hat\mu_i)^3} - \frac{{m''}_{\underline{\hat{\bf R}}_N}(\hat\mu_i)^2}{4 {m'}_{\underline{\hat{\bf R}}_N}(\hat\mu_i)^4}.\\
\end{aligned}$$
Hence the residue:
$$
\mathrm{Res}(r,\hat\mu_i) = \frac{ m_{\underline{\bf
      R}_N}'''(\hat\mu_i)}{6 m_{\underline{\bf R}_N}'(\hat\mu_i)^3}-\frac{
  m_{\underline{\bf R}_N}'' (\hat\mu_i)^2}{4 m_{\underline{\bf
      R}_N}'(\hat\mu_i)^4} \ .
$$
There are two other residues that should be taken into account for the
computation of the integral: The residues of $\xi$ on ${\mathcal
  R}_{z_2}$, and the residue for $z_1=z_2$. The first case can be handled as before. For $z_1=z_2$, the calculus of $g(z_1,z_2)$ for the residue $z_1=z_2$ is exactly the same as before. It remains to compute $\frac{1}{(z_1-z_2)^2} \frac{1}{m_{\underline{\hat{\bf R}}_N}(z_1) m_{\underline{\hat{\bf R}}_N}(z_2)}$ for the residue $z_1=z_2$.
The integration by parts yields that:

$$\oint \frac{1}{(z_1-z_2)^2} \frac{dz_1}{m_{\underline{\hat{\bf R}}_N}(z_1) m_{\underline{\hat{\bf R}}_N}(z_2)} = \oint -\frac{m'_{\underline{\hat{\bf R}}_N}(z_1)}{(z_1-z_2)} \frac{dz_1}{m_{\underline{\hat{\bf R}}_N}^2(z_1) m_{\underline{\hat{\bf R}}_N}(z_2)}.$$
Then the residue for $z_1=z_2$ is:

$$-\frac{ m'_{\underline{\hat{\bf{R}}}_N}(z_2)}{m_{\underline{\hat{\bf{R}}}_N}^3(z_2)}.$$
Again, this is the derivative function of $\frac{1}{2 m_{\underline{\hat{\bf{R}}}_N}^2(z_2)}$, then the integration is zero.

Finally both have a null contribution, hence the formula:
\begin{multline*}
\hat{\Theta}_{kk} = \frac{M^2}{ N_{k}^2}\left[
     \sum_{(i,j)\in\mathcal N_{k}^2,\ i \neq j}
\frac{-1}{(\hat\mu_i-\hat\mu_j)^2 m_{\underline{\hat{\bf R}}_N}'(\hat\mu_i) m_{\underline{\hat{\bf R}}_N}'(\hat\mu_j)}\right.
\left.+ \sum_{i\in\mathcal N_{k}} \left( \frac{ m_{\underline{\hat{\bf R}}_N}'''(\hat\mu_i)}{6 m_{\underline{\hat{\bf R}}_N}'(\hat\mu_i)^3}-\frac{ m_{\underline{\hat{\bf R}}_N}'' (\hat\mu_i)^2}{4 m_{\underline{\hat{\bf R}}_N}'(\hat\mu_i)^4}\right) \right]\ .
\end{multline*}

\end{proof}

\section{Performance in the context of cognitive radios}
\label{sec:application-cognitive}

We introduce below a practical application of the above result to the telecommunication field of cognitive radios. Consider a communication network implementing orthogonal code division multiple access (CDMA) in the uplink, which we refer to as the {\it primary} network. The primary network is composed of $K$ transmitters. The data of transmitter $k$ are modulated by the $n_k$ orthogonal $N$-chip codes ${\bf w}_{k,1},\ldots,{\bf w}_{k,n_k} \in\mathbb{C}^N$. Consider also a secondary network, in sensor mode, that we assume time-synchronized with the primary network, and whose objective is to determine the distances of the primary transmitters in order to optimally reuse the frequencies used by the primary transmitters\footnote{the rationale being that far transmitters will not be interfered with by low power communications within the secondary network.}. From the viewpoint of the secondary network, primary user $k$ has power $P_k$. Then, at symbol time $m$, any secondary user receives the $N$-dimensional data vector

\begin{equation}
  \label{eq:system_model}
  {\bf y}^{(m)} = \sum_{k=1}^K \sqrt{P_k} \sum_{j=1}^{n_k} {\bf w}_{k,j} x_{k,j}^{(m)} + \sigma {\bf n}^{(m)}
\end{equation}
with $\sigma{\bf n}^{(m)}\in\mathbb{C}^N$ an additive white Gaussian noise $\mathcal{CN}(0,\sigma^2 {\bf I})$ received at time $m$ and $x_{k,j}^{(m)} \in \mathbb{C}$ the signal transmitted by user $k$ on the carrier code $j$ at time $m$, which we assume $\mathcal{CN}(0, 1)$ as well. The propagation channel is considered frequency flat on the CDMA transmission bandwidth. We do not assume that the sensor knows $\sigma^2$ neither the vectors ${\bf w}_{k,j}$. The secondary users may or may not be aware of the number of codewords employed by each user.

Equation \eqref{eq:system_model} can be compacted under the form
\begin{equation*}
  {\bf y}^{(m)} = {\bf W}{\bf P}^{\frac{1}{2}} {\bf x}^{(m)} + \sigma {\bf n}^{(m)}
\end{equation*}
with ${\bf W}=[{\bf w}_{1,1},\ldots,{\bf w}_{1,n_1},{\bf w}_{2,1},\ldots,{\bf w}_{K,n_K}]\in\mathbb{C}^{N\times n}$, $n\triangleq \sum_{k=1}^Kn_k$, ${\bf P}\in\mathbb{C}^{n\times n}$ the diagonal matrix with entry $P_1$ of multiplicity $n_1$, $P_2$ of multiplicity $n_2$, etc. and $P_K$ of multiplicity $n_K$, and ${\bf x}^{(m)}=[{\bf x}_1^{(m)T},\ldots,{\bf x}_K^{(m)T}]^{T}\in\mathbb{C}^n$ where ${\bf x}_k^{(m)}\in\mathbb{C}^{n_k}$ is a column vector with $j$-th entry $x_{k,j}^{(m)}$.
\vspace{4mm}

Gathering $M$ successive independent observations, we obtain the matrix ${\bf Y}=[{\bf y}^{(1)},\ldots,{\bf y}^{(M)}]\in\mathbb{C}^{N\times M}$ given by
\begin{equation*}
  {\bf Y} = {\bf W}{\bf P}^{\frac{1}{2}} {\bf X} + \sigma{\bf N} = \begin{bmatrix} {\bf W}{\bf P}^{\frac{1}{2}} & \sigma{\bf I}_N\end{bmatrix}\begin{bmatrix} {\bf X} \\ {\bf N} \end{bmatrix}
\end{equation*}
where ${\bf X}=[{\bf x}^{(1)},\ldots,{\bf x}^{(M)}]$ and ${\bf N}=[{\bf n}^{(1)},\ldots,{\bf n}^{(M)}]$.

The ${\bf y}^{(m)}$ are therefore independent Gaussian vectors of zero mean and covariance ${\bf R}\triangleq{\bf W}{\bf P}{\bf W}^H+\sigma^2{\bf I}_N$. Since the objective is to retrieve the powers $P_k$, while $\sigma^2$ is known, the problem boils down to finding the eigenvalues of ${\bf W}{\bf P}{\bf W}^H+\sigma^2{\bf I}_N$. However, the sensors only have access to ${\bf Y}$, or equivalently to the sample covariance matrix
\begin{equation*}
  {{\bf R}}_N \triangleq \frac1M{\bf Y}{\bf Y}^H = \frac1M\sum_{m=1}^M {\bf y}^{(m)}{\bf y}^{(m)H}.
\end{equation*}

Assuming ${{\bf R}}_N$ conveys a good appreciation of the eigenvalue clustering to the secondary user (as in Figure \ref{fig:clusters}), Theorem \ref{prop:mestre} enables the detection of primary transmitters and the estimation of their transmit powers $P_1,\ldots,P_K$; this boils down to estimating the largest $K$ eigenvalues of ${\bf W}{\bf P}{\bf W}^H+\sigma^2{\bf I}_N$, i.e. the $P_k+\sigma^2$, and to subtract $\sigma^2$ (optionally estimated from the smallest eigenvalue of ${\bf W}{\bf P}{\bf W}^H+\sigma^2{\bf I}_N$ if $n<N$). Call $\hat{P}_k$ the estimate of $P_k$.

Based on these power estimates, the secondary user can determine the optimal coverage for secondary communications that ensures no interference with the primary network. A basic idea for instance is to ensure that the closest primary user, i.e. that with strongest received power, is not interfered with. Our interest is then cast on $P_K$. Now, since the power estimator is imperfect, it is hazardous for the secondary network to state that $K$ has power $\hat{P}_K$ or to add some empirical security margin to $\hat{P}_K$. The results of Section \ref{sec:CLT} partially answer this problem.

Theorems \ref{th:CLT} and \ref{th:estimate} enable the secondary sensor to evaluate the accuracy of $\hat{P}_k$. In particular, assume that the cognitive radio protocol allows the secondary network to interfere the primary network with probability $q$ and denote $A$ the value
\begin{equation*}
	A \triangleq \inf_a \{ {\mathbb P}(P_K-\hat{P}_K>a)\leq q\}.
\end{equation*}

According to Theorem \ref{th:CLT}, for $N,M$ large, $A$ is well approximated by $\hat{\Theta}_{K,K} Q^{-1}(q)$, with $Q$ the Gaussian cumulative distribution function. If the secondary users detect a user with power $P_K$, estimated by $\hat{P}_K$, $\mathbb{P}(\hat{P}_K+A<P_K)<q$ and then it is safe for the secondary network to assume the worst case scenario where user $K$ transmits at power $\hat{P}_K+A\simeq \hat{P}_K+\hat{\Theta}_K Q^{-1}(q)$.

In Figure \ref{fig:perf}, the performance of Theorem \ref{th:CLT} is compared against $10,000$ Monte Carlo simulations of a scenario of three users, with $P_1=1$, $P_2=3$, $P_3=10$, $n_1=n_2=n_3=20$, $N=60$ and $M=600$. It appears that the limiting distribution is very accurate for these values of $N,M$. We also performed simulations to obtain empirical estimates $\hat{\Theta}_k$ of $\Theta_k$ from Theorem \ref{th:estimate}, which suggest that $\hat{\Theta}_k$ is an accurate estimator as well.

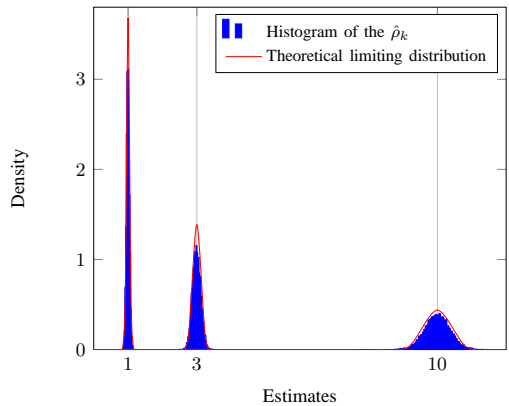
\begin{figure}
  \centering
  \begin{tikzpicture}[font=\footnotesize,scale=0.8]
    \renewcommand{\axisdefaulttryminticks}{4}
    %\pgfplotsset{every major grid/.append style={densely dashed}}
    \tikzstyle{every axis y label}+=[yshift=-10pt]
    \tikzstyle{every axis x label}+=[yshift=5pt]
    \pgfplotsset{every axis legend/.append style={cells={anchor=west},fill=white, at={(0.98,0.98)}, anchor=north east, font=\scriptsize }}
    \begin{axis}[
      %ybar,
      xmin=0,
      ymin=0,
      xmax=12,
      ymax=3.8,
      xtick={1,3,10},
      bar width=0.5pt,
      grid=major,
      ymajorgrids=false,
      scaled ticks=true,
      %scale ticks above={4},
      xlabel={Estimates},
      ylabel={Density}
      ]
      \addplot+[ybar,mark=none,color=blue,fill=blue] coordinates{
      (0.830000,0.003333)(0.860000,0.003333)(0.880000,0.063333)(0.900000,0.206667)(0.920000,0.690000)(0.940000,1.360000)(0.960000,2.393333)(0.980000,3.086667)(1.000000,3.113333)(1.020000,2.396667)(1.040000,1.713333)(1.060000,1.076667)(1.080000,0.456667)(1.100000,0.136667)(1.120000,0.040000)(1.140000,0.016667)(2.510000,0.003333)(2.620000,0.013333)(2.650000,0.006667)(2.670000,0.020000)(2.690000,0.036667)(2.710000,0.076667)(2.730000,0.073333)(2.750000,0.140000)(2.770000,0.176667)(2.790000,0.236667)(2.810000,0.320000)(2.830000,0.503333)(2.850000,0.626667)(2.870000,0.663333)(2.890000,0.773333)(2.910000,0.906667)(2.930000,0.946667)(2.950000,1.086667)(2.970000,1.050000)(2.990000,1.153333)(3.010000,1.076667)(3.030000,0.973333)(3.050000,1.023333)(3.070000,0.846667)(3.090000,0.813333)(3.110000,0.703333)(3.130000,0.586667)(3.150000,0.450000)(3.170000,0.353333)(3.190000,0.306667)(3.210000,0.243333)(3.230000,0.160000)(3.250000,0.156667)(3.270000,0.093333)(3.290000,0.050000)(3.310000,0.040000)(3.330000,0.016667)(3.350000,0.006667)(3.370000,0.013333)(3.390000,0.003333)(3.430000,0.003333)(3.460000,0.003333)(8.750000,0.003333)(8.840000,0.006667)(8.910000,0.010000)(8.940000,0.006667)(8.970000,0.003333)(8.990000,0.010000)(9.010000,0.006667)(9.040000,0.020000)(9.060000,0.003333)(9.080000,0.006667)(9.100000,0.010000)(9.120000,0.010000)(9.140000,0.010000)(9.160000,0.026667)(9.180000,0.023333)(9.200000,0.010000)(9.220000,0.033333)(9.240000,0.046667)(9.260000,0.053333)(9.280000,0.053333)(9.300000,0.060000)(9.320000,0.040000)(9.340000,0.060000)(9.360000,0.076667)(9.380000,0.083333)(9.400000,0.110000)(9.420000,0.076667)(9.440000,0.063333)(9.460000,0.150000)(9.480000,0.140000)(9.500000,0.136667)(9.520000,0.163333)(9.540000,0.146667)(9.560000,0.166667)(9.580000,0.190000)(9.600000,0.193333)(9.620000,0.223333)(9.640000,0.173333)(9.660000,0.263333)(9.680000,0.203333)(9.700000,0.266667)(9.720000,0.276667)(9.740000,0.326667)(9.760000,0.256667)(9.780000,0.333333)(9.800000,0.306667)(9.820000,0.346667)(9.840000,0.360000)(9.860000,0.363333)(9.880000,0.350000)(9.900000,0.356667)(9.920000,0.380000)(9.940000,0.366667)(9.960000,0.390000)(9.980000,0.376667)(10.000000,0.380000)(10.020000,0.353333)(10.040000,0.393333)(10.060000,0.390000)(10.080000,0.396667)(10.100000,0.360000)(10.120000,0.293333)(10.140000,0.316667)(10.160000,0.363333)(10.180000,0.346667)(10.200000,0.293333)(10.220000,0.320000)(10.240000,0.250000)(10.260000,0.333333)(10.280000,0.256667)(10.300000,0.216667)(10.320000,0.263333)(10.340000,0.220000)(10.360000,0.236667)(10.380000,0.203333)(10.400000,0.193333)(10.420000,0.203333)(10.440000,0.156667)(10.460000,0.140000)(10.480000,0.173333)(10.500000,0.153333)(10.520000,0.136667)(10.540000,0.106667)(10.560000,0.120000)(10.580000,0.093333)(10.600000,0.090000)(10.620000,0.063333)(10.640000,0.086667)(10.660000,0.060000)(10.680000,0.053333)(10.700000,0.046667)(10.720000,0.036667)(10.740000,0.033333)(10.760000,0.056667)(10.780000,0.030000)(10.800000,0.030000)(10.820000,0.013333)(10.840000,0.026667)(10.860000,0.016667)(10.880000,0.013333)(10.900000,0.023333)(10.920000,0.016667)(10.940000,0.020000)(10.970000,0.013333)(10.990000,0.006667)(11.010000,0.016667)(11.030000,0.006667)(11.050000,0.003333)(11.090000,0.003333)(11.120000,0.003333)(11.160000,0.003333)(11.300000,0.003333)
      };
      \addplot[smooth,red,line width=0.5pt] plot coordinates{
      (0.000000,0.000000)(0.083333,0.000000)(0.166667,0.000000)(0.250000,0.000000)(0.333333,0.000000)(0.416667,0.000000)(0.500000,0.000000)(0.583333,0.000000)(0.666667,0.000000)(0.750000,0.000000)(0.833333,0.002158)(0.916667,0.573164)(1.000000,3.684245)(1.083333,0.573164)(1.166667,0.002158)(1.250000,0.000000)(1.333333,0.000000)(1.416667,0.000000)(1.500000,0.000000)(1.583333,0.000000)(1.666667,0.000000)(1.750000,0.000000)(1.833333,0.000000)(1.916667,0.000000)(2.000000,0.000000)
      };

      \addplot[smooth,red,line width=0.5pt] plot coordinates{
      (0.000000,0.000000)(0.250000,0.000000)(0.500000,0.000000)(0.750000,0.000000)(1.000000,0.000000)(1.250000,0.000000)(1.500000,0.000000)(1.750000,0.000000)(2.000000,0.000000)(2.250000,0.000000)(2.500000,0.000105)(2.750000,0.129433)(3.000000,1.386399)(3.250000,0.129433)(3.500000,0.000105)(3.750000,0.000000)(4.000000,0.000000)(4.250000,0.000000)(4.500000,0.000000)(4.750000,0.000000)(5.000000,0.000000)(5.250000,0.000000)(5.500000,0.000000)(5.750000,0.000000)(6.000000,0.000000)
      };

      \addplot[smooth,red,line width=0.5pt] plot coordinates{
      (0.000000,0.000000)(0.833333,0.000000)(1.666667,0.000000)(2.500000,0.000000)(3.333333,0.000000)(4.166667,0.000000)(5.000000,0.000000)(5.833333,0.000000)(6.666667,0.000000)(7.500000,0.000000)(8.333333,0.000012)(9.166667,0.031671)(10.000000,0.437698)(10.833333,0.031671)(11.666667,0.000012)(12.500000,0.000000)(13.333333,0.000000)(14.166667,0.000000)(15.000000,0.000000)(15.833333,0.000000)(16.666667,0.000000)(17.500000,0.000000)(18.333333,0.000000)(19.166667,0.000000)(20.000000,0.000000)
      };
      \legend{Histogram of the $\hat{\rho}_k$ \\ Theoretical limiting distribution \\}
    \end{axis}
  \end{tikzpicture}
  \caption{Comparison of empirical against theoretical variances, based on Theorem \ref{th:CLT}, for three users, $P_1=1$, $P_2=3$, $P_3=10$, $n_1=n_2=n_3=20$ codes per user, $N=60$, $M=600$ and SNR$=20$ dB.}
  \label{fig:perf}
\end{figure}

\vspace{4mm}
\section{Conclusion}
\label{sec:conclusion}
In this article, we derived an exact expression and an approximation
of the limiting performance of a statistical inference method that
estimates the population eigenvalues of a class of sample covariance
matrices. These results are applied in the context of cognitive radios
to optimize secondary network coverage based on measures of the
primary network activity.

\vspace{8mm}

\begin{appendix} \label{app:proofs}

\subsection{Proof of Proposition \ref{prop:eigen-outside-bulk}} \label{app:eigen-outside-bulk}

Let us first begin by considerations related to the supports of the
probability distributions associated to $m(z)$ and $m_N(z)$. Denote by
${\mathcal S}$ and ${\mathcal S}_N$ these supports and recall that
${\mathcal S}$ is the union of $L$ clusters:
$$
{\mathcal S}= (a_1,b_1)\cup \cdots \cup (a_L,b_L)\ .
$$
The following proposition clarifies the relations between ${\mathcal S}_N$ and ${\mathcal S}$.

\begin{proposition}\label{prop:support} Let $N,M\rightarrow\infty$, then for $N$ large enough, the support ${\mathcal S}_N$
  of the probability distribution associated to the Stieltjes
  transform $m_N(z)$ is the union of $L$ clusters:
$$
{\mathcal S}_N = (a^N_1,b^N_1)\cup \cdots \cup (a^N_L,b^N_L)\ .
$$
Moreover, the following convergence holds true :
$$
a_\ell^N \xrightarrow[N,M\rightarrow\infty]{} a_\ell\ ,\quad b_\ell^N \xrightarrow[N,M\rightarrow\infty]{} b_\ell\ ,
$$
for $1\le \ell \le L$.
\end{proposition}

\begin{remark}
If the support ${\mathcal S}_N$ contains zero, (ex: $N>M$), then zero is also in the support $\mathcal S$, the conclusion is still true.

\end{remark}

\begin{proof}[Proof of Proposition \ref{prop:support}] Recall the relations:
\begin{equation}\label{eq:defeqdet}
  \underline{m}_N(z) = - \left( z - \frac{N}M\int \frac{t}{1+t\underline{m}_N(z)}dF^{{\bf R}_N}(t) \right)^{-1}\
\end{equation}
and

\begin{equation}
m_N(z) = \frac{M}{N} \underline{m}_N(z)-\left( 1-\frac{M}{N} \right) \frac{1}{z}.
\end{equation}

As the inverse of Stieltjes transform of $-\frac{1}{z}$ is $\delta_0$ (the Dirac mass on 0) and $\underline{m}_N(z)$ is a continuous function over $\mathbb{R}_{+}^{*}$, for $a$,$b$ with $0<a<b$, by the inverse formula of Stieltjes transform, one gets:

$$F_N([a,b])= \frac{M}{N}\underline{F}_N([a,b]).$$
So it suffices to study the support $\underline{\mathcal{S}}_N$ associated to $\underline{F}_N$.

From the definition of $\underline{m}_N(z)$ (see formula (\ref{eq:defeqdet})), we obtain:
$$
z_{{\bf R}_N}(\underline{m}_N)=-\frac{1}{\underline{m}_N}+ \frac{N}{M} \int \frac{t d F^{\bf R_N}(t)}{1+t \underline{m}_N(z)}.$$
Denote by $B=\{m \in \mathbb{R}: m \neq 0, -m^{-1} \notin \{\rho_1,\cdots,\rho_L\} \}$. In \cite[Theorem 4.1 and Theorem 4.2]{CHO95}, Silverstein and Choi show that for a real number $x$, $x \in \underline{\mathcal S}_N^c \Longleftrightarrow \underline{m}_x \in B$ and $z_{\bf R_N}'(\underline{m}_x)= \frac1{\underline{m}_x^2} - \frac{N}M \int \frac{t^2 d F^{\bf R_N}(t)}{(1+t \underline{m}_x)^2}>0$ with $\underline{m}_N(x)=\underline{m}_x$ and $z_{\bf R_N}(\underline{m}_  x) = x$.

Then if $a \in \partial{\underline{\mathcal S}}_N$, $m_a \notin B$ or $z_{\bf R_N}'(m_a) \leq 0$ with $m_a=\underline{m}_N(a)$. Now we will show that $m_a \in B$. In \cite[Theorem 5.1]{CHO95}, $m_a \neq 0$. If $-m_a^{-1} \in S_{F^{\bf R_N}}$, as $F^{\bf R_N}$ is discrete, we get that $\lim_{m \rightarrow m_a} \int \frac{t^2 d F^{\bf R_N}(t)}{(1+t m)^2} \longrightarrow \infty$. So on the neighborhood to the left and to the right of $m_a$, $z_{\bf R_N}' <0$ which contradicts \cite[Theorem 5.1]{CHO95}.

Hence $z_{\bf R_N}'(m_a) \leq 0$. By the continuity, we get

$$z_{\bf R_N}'(m_a) = \frac1{m_a^2} - \frac{N}M \int \frac{t^2 d F^{{\bf R}_N}(t)}{(1+t m_a)^2}=0.$$
It is equivalent to the following equation:
\begin{equation}
\label{eq:borne}
z_{\bf R_N}'(m_a) = \frac1{m_a^2} - \frac{1}M \sum_{i=1}^{L} N_i \frac{\rho_i^2} {(1+\rho_i m_a)^2}=0.
\end{equation}
By multiplying the common denominator, one will get a polynomial of the degree 2L in $m_a$.
Now we will show that these 2L roots are real. At first, notice that

$$\frac1{m^2} - \frac{N}M \int \frac{t^2 d F^{{\bf R}_N}(t)}{(1+t m)^2} \xrightarrow[m \rightarrow -\frac{1}{\rho_i}]{} - \infty ,$$
and
$$z_{\bf R_N}''(m) = -\frac2{m^3} + \frac{N}M \int \frac{2 t^3 d F^{{\bf R}_N}(t)}{(1+t m)^3}. $$
So $z_{\bf R_N}''(m)$ has one and only one zero in the open set $(-\frac{1}{\rho_i}, -\frac{1}{\rho_{i+1}})$ for $i\in \{1,\cdots,L-1\}$. Then for $\beta_i \in (-\frac{1}{\rho_i}, -\frac{1}{\rho_{i+1}})$ such that $z_{\bf R_N}''(\beta_i)=0$, it suffices to show that $z_{\bf R_N}'(\beta_i) >0 $ in order to prove that there will be two zeros for $z_{\bf R_N}'(m)$ in the set $(-\frac{1}{\rho_i}, -\frac{1}{\rho_{i+1}})$. From the separability condition (cf. Assumption (A2)), $\inf_N \{ \frac{M}N - \Psi_N(i)\} >0$, and

$$\begin{aligned} z_{\bf R_N}'(-\frac{1}{\alpha_i}) &= \alpha_i^2 - \frac{N}M \int \frac{t^2 d F^{\bf R_N}(t)}{(1-\frac{t}{\alpha_i})^2}\\
&=\alpha_i^2 \left( 1-\frac{1}{M} \sum_{r=1}^L N_i \frac{\rho_i^2 }{(\alpha_i-\rho_i)^2}         \right) >0\\
\end{aligned}$$
Thus we obtain $2(L-1)$ roots. Besides, in the open set $(-\rho_L^{-1}, 0)$, $$\frac1{m^2} - \frac{N}M \int \frac{t^2 d F^{{\bf R}_N}(t)}{(1+t m)^2} \xrightarrow[m_a \rightarrow 0^{-}]{} +\infty,$$
there exists another root in this set. In the open set $(-\infty, -\rho_1^{-1})$, $$\frac1{m^2} - \frac{N}M \int \frac{t^2 d F^{{\bf R}_N}(t)}{(1+t m)^2} \xrightarrow[m \rightarrow -\infty]{} 0$$
and $$\frac1{m^2} - \frac{N}M \int \frac{t^2 d F^{{\bf R}_N}(t)}{(1+t m)^2}   \underset{m \rightarrow -\infty}{\sim} \frac1{m^2}(1-\frac{L}{M}) >0. $$
Hence the last root in this open set. This proves that ${\mathcal S}_N = (a^N_1,b^N_1)\cup \cdots \cup (a^N_L,b^N_L) .$

To prove $a_\ell^N \xrightarrow[N,M\rightarrow\infty]{} a_\ell$ and $b_\ell^N \xrightarrow[N,M\rightarrow\infty]{} b_\ell$  , notice that $a_i$ $b_i$ satisfy the same type of the equation by replacing $\frac{N}{M}$ by $c$ and $F^{\bf R_N}$ by $F^{\bf R}$. As $\frac{N}{M} \rightarrow c$ and $\frac{K_i}{M} \rightarrow c_i$, the roots of Equation (\ref{eq:borne}) converge to those of the limit equation (see \cite{CU89}). Thus we achieve the second conclusion.

\end{proof}

We are now in position to establish the proof of Proposition
\ref{prop:eigen-outside-bulk}.

Denote by ${\mathcal S}(\varepsilon)$ the $\varepsilon$-blow-up of ${\mathcal S}$, i.e.
${\mathcal  S}(\varepsilon) =\{x\in \mathbb{R},\ d(x,S)<\varepsilon\}$.
Let $\varepsilon>0$ be small enough and consider a smooth function
$\phi$ equal to zero on ${\mathcal  S}(\varepsilon/3)$, equal
to 1 if $x \notin {\mathcal S}(\varepsilon)$, equal to zero again if $|x|\ge \tau$ (as we shall see,
$\tau$ will be chosen to be large), and smooth in-between with
$0\le \phi\le 1$:
$$
\phi(x)=\begin{cases}
0 & \text{if}\  d(x,\mathcal{S})<\varepsilon/3\ , \\
1 &\text{if}\ d(x,\mathcal{S})> \varepsilon\ , |x| \leq \tau-\epsilon\\
0 & \text{if}\ |x|> \tau\ .
\end{cases}
$$
Notice that if $N,M\rightarrow\infty$ and $N$ is large enough, then by
Proposition \ref{prop:support}, $\phi(x)=0$ for all $x\in {\mathcal
  S}_N$.  Now if ${\bf Z}$ is a $M\times M$ hermitian matrix with
spectral decomposition ${\bf Z}= {\bf U}\, \mathrm{diag}\left(
  \gamma_i;\ 1\le i\le M) \right)\, {\bf U}^H$, where ${\bf U}$ is
unitary and $\mathrm{diag}\left( \gamma_i;\ 1\le i\le M) \right)$
stands for the $M\times M$ diagonal matrix whose entries are ${\bf
  Z}$'s eigenvalues, write $\phi({\bf Z}) = {\bf U}\,
\mathrm{diag}\left( \phi(\gamma_i);\ 1\le i\le M) \right)\, {\bf
  U}^H$.

We have:
\begin{eqnarray*}
  \mathbb{P}(\sup_n d(\lambda_n,S)> \varepsilon ) &\leq& \mathbb{P} (\|\hat{\bf R}_N\|  > \tau- \varepsilon)\ +\ \mathbb{P}(\mathrm{Tr}\,\phi(\hat {\bf R}_N)\geq 1) \\
  &=& \mathbb{P} (\|\hat{\bf R}_N\| > \tau - \varepsilon)\ +\ \mathbb{P}([\mathrm{Tr}\,\phi(\hat{\bf R}_N)]^{p} \geq 1) \\
&\stackrel{(a)}\leq&\mathbb{P} (\|\hat{\bf R}_N\| > \tau -\varepsilon)\ +\
  \mathbb{E}[\mathrm{Tr}\,\phi(\hat{\bf R}_N)]^{p}\ ,
\end{eqnarray*}
for every $p\ge 1$, where $(a)$ follows from Markov's inequality.  The
fact that $\mathbb{P} (\|\hat{\bf R}_N\| > \tau)= \OO(N^{-\ell})$ for
$\tau$ large enough and every $\ell \in \mathbb{N}^*$ is well-known
(see for instance \cite[Section 9.7]{SIL06}). We shall therefore
establish estimates over $\mathbb{E}[\mathrm{Tr}\,\phi(\hat{\bf
  R}_N)]^{p}$. Take $p=2^k$; we prove the following statement by
induction: For $k\ge 1$ and for every integer $\beta < 2^k$ and for
every smooth function $f$ with compact support whose value on
${\mathcal S}({\varepsilon}/3)$ is zero ,
$$
\mathbb{E}\left( \mathrm{Tr} f(\hat{\bf R}_N)\right)^{2^k}\quad =\quad  \OO\left(\frac{1}{N^{\beta}}\right)\ .
$$

First notice that due to Proposition \ref{prop:support},
$\int_{{\mathcal S}_N} f(\lambda) \, F_N(d\lambda)=0$ (where $F_N$
is the probability distribution associated to $m_N$) for $N,M$ large
enough ($N,M\rightarrow\infty$). A minor modification of \cite[Lemma
2]{LOU10} (whose model is slightly different) with the help of
\cite[Proposition 5]{HAC06} yields that for $N,M\rightarrow\infty$ and
$N$ large enough, $\mathbb{E}\, \mathrm{Tr}\,f(\hat {\bf R}_N)=\OO
({N}^{-1})$, and the property is verified for $k=0$.

%%%%%%%%%%%%%%%%%%%%%%%%%%%%%%%%%%%%%%%%%%%%%%%%%%%%%%%%%%%%%%%%%%%%%%%%%%%%

Let $k>0$ be fixed and assume that the result holds true for
$\beta<2^k$. We want to show that $\mathbb{E}[\mathrm{Tr} f(\hat{\bf
  R}_N)]^{2^{(k+1)}}= \OO (N^{-2 \beta})$. At step $k+1$, the
expectation writes:
\begin{eqnarray}\label{eq:induction}
  \lefteqn{\left| \mathbb{E}[\mathrm{Tr}\, f(\hat{\bf R}_N)]^{2^{(k+1)}} \right|}
  \\ &=& \left| \mathbb{E}\left([\mathrm{Tr} f(\hat{\bf R}_N)]^{2^k} + \mathbb{E}[\mathrm{tr} f(\hat{\bf R}_N)]^{2^k}- \mathbb{E}[\mathrm{Tr} f(\hat{\bf R}_N)]^{2^k}\right)^2 \right|\nonumber \\
  &\leq& 2\left(\mathrm{Var}[\mathrm{Tr}\, f(\hat{\bf R}_N)]^{2^k}+
|\mathbb{E}[\mathrm{Tr} f(\hat{\bf R}_N)]^{2^k}|^2\right).
\end{eqnarray}
The second term of the right hand side (r.h.s.) of the equation can be handled by the induction hypothesis:
$$
\left|\mathbb{E}[\mathrm{Tr} f(\hat{\bf R}_N)]^{2^k}\right|^2=\OO\left(\frac{1}{N^{2\beta}}\right).
$$
We now rely on Poincar\'e-Nash inequality (see for instance
\cite[Section II-B]{HAC06}) to handle the first term of the
r.h.s. Applying this inequality, we obtain:
\begin{eqnarray}\label{eq:PN}
\mathrm{Var} \left( (\tr f(\hat{\bf R}_N))^{2^k} \right)
&\le &
K \sum_{i,j}
\mathbb{E} \left[ \left|\frac{\partial [\tr f(\hat{\bf R}_N)]^{2^k}}{\partial Y_{i,j}}  \right|^2
+ \left|\frac{\partial [\tr f( \hat{\bf R}_N)]^{2^k}}{\partial \overline{Y}_{i,j}} \right|^2\right]\ ,
\end{eqnarray}
where $K$ is a constant which does not depend on $N,M$ and which is
greater than ${\bf R}_N$'s eigenvalues. In order to compute the
derivatives of the r.h.s., we rely on \cite[Lemma
4.6]{HAA06}. This yields:
\begin{eqnarray*}
\frac{\partial}{\partial Y_{i,j}} [\tr f( \hat{\bf R}_N)]^{2^k}&=& \frac{2^k}{M}
[\tr f(\hat{\bf R}_N)]^{2^k-1} [{\bf Y}_N^{*}  f'(\hat{\bf R}_N)]_{j,i}\ ,\\
\frac{\partial}{\partial \overline{Y_{i,j}}} [\tr f(\hat{\bf R}_N)]^{2^k}&=&
\frac{2^k}{M}  [\tr f(\hat{\bf R}_N)]^{2^k-1} [f'(\hat{\bf R}_N) {\bf Y}_N]_{i,j}\ .
\end{eqnarray*}
Plugging these derivatives into \eqref{eq:PN}, we obtain:
\begin{eqnarray*}
\lefteqn{\mathrm{Var}(\mathrm{Tr} [f(\hat{\bf R}_N)]^{2^k})}\\
& \leq& \frac{K\, 2^{2k+1}}{M^2} \mathbb{E} \left[ (\tr f(\hat{\bf R}_N))^{(2^{k+1}-2)}\,
\tr (f'(\hat{\bf R}_N) {\bf Y}_N {\bf Y}_N^{*} f'(\hat{\bf R}_N) ) \right] \ ,\\
& = & \frac{K\, 2^{2k+1}}{M} \mathbb{E} \left[ (\tr f(\hat{\bf R}_N))^{(2^{k+1}-2)}\,
\tr (f'(\hat{\bf R}_N)^2 \hat{\bf R}_N) \right]    \ ,   \\
& \leq& \frac{K\, 2^{2k+1}}{M} \left| \mathbb{E} [\mathrm{Tr} f(\hat{\bf R}_N)]^{2^{k+1}} \right|^{\frac{2^{k+1}-2}{2^{k+1}}}
\times  \left|
\mathbb{E} [\mathrm{Tr} f'(\hat{\bf R}_N)^2 \hat{\bf R}_N]^{2^k} \right| ^{\frac{1}{2^k}}\ ,
\end{eqnarray*}
where the last inequality is a consequence of H\"{o}lder's inequality.

As the function $h(\lambda)=\lambda [f'(\lambda)]^2$ satisfies the induction hypothesis, we have for every
$\alpha<1$:
$$
\left| \mathbb{E}\mathrm{Tr}[ f'(\hat{\bf R}_N)^2 \hat{\bf R}_N]^{2^k} \right| ^{\frac{1}{2^k}}=\OO(N^{-\alpha}).
$$
Plugging this estimate into \eqref{eq:induction}, we obtain:
\begin{equation}
\label{eq:tight}
\left| \mathbb{E}[\tr f(\hat{\bf R}_N)]^{2^{(k+1)}}\right|  \leq K \left( \frac{1}{N^{1+\alpha}} |\mathbb{E}[\tr
f(\hat{\bf R}_N)]^{2^{(k+1)}}|^{\frac{2^{k+1}-2}{2^{k+1}}} \right) + \OO(N^{-2\beta})\ ,
\end{equation}
where $K$ is a constant independent of $M,N,k$. Notice that inequality
\eqref{eq:tight} involves twice the quantity of interest
$\mathbb{E}[\tr f(\hat{\bf R}_N)]^{2^{(k+1)}}$ that we want to upper
bound by $O(N^{-2\beta})$. We shall proceed iteratively.

Notice that $\tr [ f(\hat{\bf R}_N)] \leq \sup_{x\in \mathbb{R}}
|f(x)| \times N$ because $f$ is bounded on $\mathbb R$; hence the rough estimate:
$$\mathbb{E}[\mathrm{Tr} f(\hat{\bf R}_N)]^{2^{(k+1)}} = \OO
(N^{2^{k+1}}).
$$
Plugging this into \eqref{eq:tight} yields:
$$
\mathbb{E}[\tr f(\hat{\bf
  R}_N)]^{2^{(k+1)}} \ =\ \OO (N^{a_1})\ ,
$$
where $a_0=2^{k+1}$ and $a_1=a_0 \frac{2^{k+1}-2}{2^{k+1}} - (1+\alpha).$ Iterating the procedure,
we obtain:
$$
\mathbb{E}[\tr f(\hat{\bf
  R}_N)]^{2^{(k+1)}} \ =\  \OO \left( N^{a_\ell \vee (-2\beta)}\right)\ ,
$$
where $a_\ell=a_{\ell-1} \frac{2^{k+1}-2}{2^{k+1}}-(1+ \alpha)$ and
$x\vee y$ stands for $\sup(x,y)$. Now, in order to conclude the
proof, it remains to prove that i) the sequence $(a_\ell)$ converges
to some limit $a_\infty$, ii) for some well-chosen $\alpha<1$,
$a_\infty \in (-2^{k+1}, -2\beta)$. Write:
$$
a_{\ell+1}+ 2^k (1+\alpha) = \frac{2^{k}-1}{2^{k}}(a_\ell +2^k (1+\alpha ) )\ ,
$$
hence $a_\ell$ converges to $-2^k(1+\alpha)$ which readily belongs to
$(-2^{k+1}, -2\beta)$ for a well-chosen $\alpha\in (0,1)$.
Finally $\mathbb{E}[\tr f(\hat{\bf R}_N)]^{2^{(k+1)}} \ =\ \OO (N^{-2 \beta})$
which ends the induction.

% It now remains to obtain a similar estimate for the function $\phi$
% which has not a compact support.

% We now prove that we also have for each $k$ and all $\beta < 2^k$
% $$\mathbb{E}[\mathrm{Tr} \phi({\bf R}_N)]^{2^k}= O(\frac{1}{N^{\beta}}).$$

% \noindent The case $k=0$ is treated in (\cite{LOU10}, lemma 2). Assuming the validity of the induction at step $k$, with the same method, we obtain
% $$\begin{aligned} & \textrm{Var}(\mathrm{Tr} \phi({\bf R}_N)^{2^k}) \\ & \leq \frac{K 2^{2k+1}}{M} |\mathbb{E} (\mathrm{Tr}\phi({\bf R}_N))^{2^{k+1}}| ^{\frac{2^{k+1}-2}{2^{k+1}}} |\mathbb{E}[\mathrm{Tr} \phi'({\bf R}_N)^2 {\bf R}_N]^{2^k}|^{\frac{1}{2^k}} \end{aligned}$$
% with $K$ the same constant defined as before. The function $\lambda [\phi'(\lambda)]^2$ has compact support and is zero on $\mathcal S$, so that we have for all $\alpha <1$
% $$|\mathbb{E}\mathrm{Tr}[ \phi'({\bf R}_N)^2 {\bf R}_N]^{2^k}|^{\frac{1}{2^k}}=O(N^{-\alpha}).$$
% Then
% $$\mathbb{E}[\mathrm{Tr} \phi({\bf R}_N)]^{2^{(k+1)}} \leq O \left( \frac{1}{M^{1+ \alpha}} |\mathbb{E}[\mathrm{Tr} \phi({\bf R}_N)]^{2^{(k+1)}}|^{\frac{2^{k+1}-2}{2^{k+1}}} \right) + O(N^{-2 \beta}).$$
% As $\phi$ is bounded on $\mathbb R$, following similar steps as before, we easily end the proof.

It remains to apply this estimate to $\mathbb{E}[\tr \phi(\hat{\bf
  R}_N)]^{\ell}$ in order to get the desired result.

\subsection{Proof of Lemma \ref{lemma:convergence-approx}}\label{app:convergence-approx}
As explained in Section \ref{sec:CLT}, there are two conditions to prove (Billingsley \cite[Theorem 13.1]{BIL08}):
\begin{itemize}
\begin{item} Finite-dimensional convergence of the process $(X_N, X_N')$.
\end{item}

\begin{item}
Tightness on the contour $\mathcal C_k$.
\end{item}
\end{itemize}

\begin{remark}
As $u_N$ (resp. $u'_N$) converges almost surely to $\underline{u}$ (resp. $\underline{u}'$) (see Silverstein and Bai \cite{SIL95}), the convergence of the process $(X_N, X_N', u_N, u_N')$ is achieved as soon as the convergence of the process $(X_N, X_N')$ is proved.

\end{remark}

In \cite{BAI04}, Bai and Silverstein establish a central limit theorem for $F^{{\bf R}_N}$ with the complex Gaussian entries $X_{ij}$. We recall below their main result.
\begin{proposition}\cite{BAI04}
	\label{prop:baisil}
With the notations introduced in Section \ref{sec:model}, for $f_1,\ldots,f_p$, analytic on an open region containing $\mathbb{R}$,

\begin{enumerate}

\begin{item}
$\left(N\int f_i(x)d(F^{\hat{\bf R}_N}-F_N)(x)\right)_{1\leq i\leq p}$ forms a tight sequence on $N$,
\end{item}

\begin{item}

$$\left(N\int f_i(x)d(F^{\hat{\bf R}_N}-F_N)(x)\right)_{1\leq i\leq p} \xrightarrow[]{\mathcal D} \mathcal N(0,\bf V), \nonumber$$
where ${\bf V}=(V_{ij})$ and
$$V_{ij} = -\frac1{4\pi^2}\oint\oint f_i(z_1)f_j(z_2)v_{ij}(z_1,z_2) dz_1dz_2, \nonumber $$
with
$$v_{ij}(z_1,z_2) = \frac{\underline{m}'(z_1)\underline{m}'(z_2)} {(\underline{m}(z_1)-\underline{m}(z_2))^2}-\frac1{(z_1-z_2)^2}$$
where the integration is over positively oriented contours that circle around the support $\mathcal S$.
\end{item}
\end{enumerate}
\end{proposition}

Now we apply this proposition to show the finite-dimensional convergence. For all $z_i \in \mathcal C_k \backslash \mathbb{R}$, notice that
$$m_{\hat{\bf R}_N}(z)- m_N(z) = \frac{1}{2 i \pi} \oint \frac{1}{x-z} d(F^{\hat{\bf R}_N}-F_N)(x) $$ with the contour who contains the support $\mathcal S$ and $X_N(z)= M(m_{\underline{\hat{\bf R}}_N}(z)-\underline{m}_N(z))$. Then Proposition \ref{prop:baisil} implies directly that for all $p \in \mathbb{N}$, the random vector $$\Big{(}X_N(z_1) , X_N^{'}(z_1),\cdots, X_N(z_p) , X_N^{'}(z_p)\Big{)}$$ converges to a centered Gaussian vector by considering the functions:
$$\left( f_1(x)=\frac{1}{x-z_1}, f_2(x)=\frac{1}{(x-z_1)^2},\cdots,f_{2p-1}(x)=\frac{1}{x-z_p}, f_{2p}(x)=\frac{1}{(x-z_p)^2} \right).$$ Thus the finite dimensional convergence is achieved.

The proof of the tightness is based on Nash-Poincar\'e inequality (\cite{LOU10} and \cite{HAC06}). In Appendix \ref{app:eigen-outside-bulk}, it is proved that for all $\epsilon >0$ and all $\ell \in \mathbb N$,
$$\mathbb{P}\left( \sup_{\lambda \in \mathrm{eig}(\hat{\bf R}_N) } d(\lambda, \mathcal S)>\epsilon \right) = o(N^{-\ell}).$$
Following the same idea as Bai and Silverstein \cite[Section 3 and 4]{BAI04}, it is indeed a tight sequence. The details of the proof are in Appendix \ref{app:tightness}. Thus Lemma \ref{lemma:convergence-approx} is achieved.

\subsection{Proof of the tightness}
\label{app:tightness}

We will show the tightness of the sequence $M( m_{\underline{\hat{\bf R}}_N}-\underline{m}_N) $ and $M( m_{\underline{\hat{\bf R}}_N}^{'}-\underline{m}^{'}_N) $ by using Nash-Poincar\'e's inequality \cite{HAC06}.
First, denote by $M(m_{\underline{\hat{\bf R}}_N}(z) - \underline{m}_{N}(z))= M_N^{1}(z)+M_N^{2}(z)$ with $M_N^{1}(z)= M(m_{\underline{\hat{\bf R}}_N}(z) - \mathbb{E}[m_{\underline{\hat{\bf R}}_N}(z)])$ and $M_N^2(z) = M( \mathbb{E}[m_{\underline{\hat{\bf R}}_N}(z)] - \underline{m}_{N}(z)).$

As $\frac{1}{\hat{\rho}_k-z}$ can converge to infinite if z is close to the real axis, there will be a little trouble for the tightness. Then we need a truncated version of the process. More precisely, let $\varepsilon_N$ be a real sequence decreasing to zero satisfying for some $\delta \in ]0,1[$:
$$\varepsilon_ N \geq N^{-\delta}\ .$$

\begin{remark}
Notice that $X_N(z) = M( m_{\underline{\hat{\bf R}}_N}-\underline{m}_N) =\overline{X_N(\overline{z})}$ for $z\in \overline{\mathbb C^{+}}$. So it suffices to verify the arguments for $z \in \mathbb C^{+}$.
\end{remark}

Denote by $([x_{2k-1},x_{2k}], k=1,\cdots,L)$ the $k$-th cluster of the support of the
limiting spectral measure; and take $l_{2k-1}, l_{2k}$ such that $x_{2k-2} < l_{2k-1} < x_{2k-1}$ and $x_{2k} < l_{2k}< x_{2k+1} $ for $k \in \{1,..,L\}$ with conventions $x_0=0$ and $x_{2L+1}=\infty$, $\it{i.e.}$ , $[l_{2k-1}, l_{2k}]$ only contains the k-th cluster. Let $d>0$. Consider:
$$C_u=\{x+i d: x \in [l_{2k-1}, l_{2k}]\}.$$
and
$$C_{r}= \{l_{2k-1}+\mathbf{i}v: v\in [N^{-1}\varepsilon_N, d]\}.$$
Also $$C_l=\{l_{2k}+\mathbf{i}v: v\in [N^{-1}\varepsilon_N, d]\}.$$

\noindent Then $C_N= C_l \cup C_u \cup C_r.$ The process $\hat{M}_N^{1}(\cdot)$ is defined by
$$\hat{M}_N^{1}(z)=\begin{cases}
M_N^{1}(z) & \text{for $z\in C_N$,}\\
M_N^{1}(l_{2k}+ \mathbf{i} N^{-1}\varepsilon_N) & \text{ for $x=l_{2k}, v \in [0, N^{-1}\varepsilon_N]$,}\\
M_N^{1}(l_{2k-1}+ \mathbf{i} N^{-1}\varepsilon_N) & \text{ for $x=l_{2k-1}, v \in [0, N^{-1}\varepsilon_N]$.}\\
\end{cases}$$

\noindent This partition of $C_N$ is identical to that used in \cite[Section 1]{BAI04}. With probability one (see \cite{SIL98} and \cite{BAI99}), for all $\epsilon >0$, $$\lim \sup_{\lambda \in \mathrm{eig}(\hat{\bf R}_N) } d(\lambda, \mathcal{S}_N) < \epsilon$$
with $d(x,S)$ the Euclidean distance of $x$ to the set $S$.
So with probability one, for all $N$ large, (\cite[page 563]{BAI04})
$$\left| \oint \Big{(}M_N^{1}(z)-\hat{M}_N^{1}(z)\Big{)}dz \right| \leq K_1 \varepsilon_N,$$
and
$$\left| \oint \Big{(} {M_N^{1}}'(z)-{\hat{M}_N^{2'}}(z)  \Big{)} dz \right| \leq K_2 \varepsilon_N$$
for some constants $K_1$ and $K_2$. Both terms converge to zero as $M\to \infty$. Then it suffices to ensure the tightness for $\hat{M}_N^{1}(z)$ and ${\hat{M}_N^{1'}}(z)$.

\noindent We now prove tightness based on \cite[Theorem 13.1]{BIL08}, i.e.
\begin{enumerate}
\begin{item}
Tightness at any point of the contour (here $C_N$).
\end{item}
\begin{item}
	Satisfaction of the condition
 $$\sup_{N, z_1,z_2 \in C_N} \frac{\mathbb{E}|(\hat{M}_N^{1}(z_1)- \hat{M}_N^{1}(z_2))|^2}{|z_1-z_2|^2} \leq K.$$
\end{item}
\end{enumerate}
Condition 1) is achieved by an immediate application of Proposition \ref{prop:baisil}. We now verify the second condition.

\vspace{4mm}
We evaluate $\frac{\mathbb{E}|(\hat{M}_N^{1}(z_1)- \hat{M}_N^{1}(z_2))|^2} {|z_1-z_2|^2}$.
Notice that
$$\begin{aligned} m_{\underline{\hat{\bf R}}_N}(z_1)-m_{\underline{\hat{\bf R}}_N}(z_2)& =\frac{z_1-z_2}{M}\sum_{i=1}^{N} \frac{1}{(\hat{\lambda}_i-z_1)(\hat{\lambda}_i-z_2)}\\ &=\frac{z_1-z_2}{M} \mathrm{Tr}({\bf D}_N^{-1}(z_1) {\bf D}_N^{-1}(z_2))
\end{aligned}$$
with ${\bf D}_N(z)= \hat{{\bf R}}_N- z {\bf I}_N$.
\noindent We have
$$ \begin{aligned} &\frac{\partial}{\partial Y_{i,j}} \left( \frac{m_{\underline{\hat{\bf R}}_N}(z_1)-m_{\underline{\hat{\bf R}}_N}(z_2)}{z_1-z_2} \right) \\ &= \frac{\partial}{\partial Y_{i,j}} \mathrm{Tr} (\hat{\bf R}_N-z_1 I)^{-1}(\hat{\bf R}_N-z_2 I)^{-1}\\
&= \frac{1}{M}\left[-{\bf Y}_N^{*} {\bf D}_N^{-2}(z_1) {\bf D}_N^{-1}(z_2) - {\bf Y}_N^{*} {\bf D}_N^{-1}(z_1) {\bf D}_N^{-2}(z_2) \right]_{j,i},\\
\end{aligned}$$
and

$$ \begin{aligned} & \frac{\partial}{\partial \bar{Y}_{i,j}} \left( \frac{m_{\underline{\hat{\bf R}}_N}(z_1)-m_{\underline{\hat{\bf R}}_N}(z_2)}{z_1-z_2} \right) \\ &= \frac{1}{M}[-{\bf D}_N^{-2}(z_1) {\bf D}_N^{-1}(z_2) {\bf Y}_N -{\bf D}_N^{-1}(z_1) {\bf D}_N^{-2}(z_2) {\bf Y}_N]_{i,j}. \end{aligned}$$
Then by the Nash-Poincar\'e inequality and the fact that $\hat{\bf R}_N$ is uniformly bounded in spectral norm almost surely, one gets
$$\begin{aligned} & \frac{\mathbb{E}|\hat{M}_1(z_1)- \hat{M}_1(z_2)|^2}{|z_1-z_2|^2}  \leq \frac{C_1}{N} \mathbb{E}\Big{[}\mathrm{Tr}({\bf L}_N) \Big{]} \\ & = \frac{C_1}{N} \mathbb{E} (\mathrm{Tr}({\bf L}_N) \mathrm{I}_{\sup_n d(\hat{\lambda}_n, \mathcal S)\leq \varepsilon}) + \frac{C_1}{N} \mathbb{E} (\mathrm{Tr}({\bf L}_N)\mathrm{I}_{\sup_n d(\hat{\lambda}_n, \mathcal S)>\varepsilon})
\end{aligned}$$
with $${\bf L}_N=\hat{\bf R}_N {\bf D}_N^{-4}(z_1) {\bf D}_N^{-2}(z_2) + 2 \hat{\bf R}_N {\bf D}_N^{-3}(z_1) {\bf D}_N^{-3}(z_2)+ \hat{\bf R}_N {\bf D}_N^{-2}(z_1) {\bf D}_N^{-4}(z_2)$$ and $C_1$ a constant which does not depend on $N$ or $M$. For the first term, $\mathrm{Tr}({\bf L}_N)$ is bounded on the set $\sup_n d(\hat{\lambda}_n, \mathcal S)\leq \varepsilon$. For the second term, since for all $i \in \mathbb N$ and all $z \in C_N$, $\frac{1}{|\hat{\lambda}_n- z|^i} \leq \frac{N^i}{\varepsilon_N^i}$ , it leads that $$\sum_{n=1}^N \frac{1}{|\hat{\lambda}_n- z|^i} \leq \frac{N^{i+1}}{\varepsilon_N^i}.$$ Then
$$|\mathrm{Tr}({\bf L}_N)| \leq \OO \left( \frac{ N^{7}}{\varepsilon_N^6} \right).$$ As $\mathbb{P}(\sup d(\hat{\lambda}_n, \mathcal S) \geq \varepsilon)=o(N^{-16})$, take $\varepsilon_N=N^{-0.01}$, one obtains $$\begin{aligned} \left| \mathbb{E} (\mathrm{Tr}({\bf L}_N)\mathrm{I}_{\sup_n d(\hat{\lambda}_n, S)>\varepsilon}) \right| & \leq  \mathbb{E} \left| \mathrm{Tr}({\bf L}_N)I_{\sup d(\hat{\lambda}_n, S)>\varepsilon} \right| \\  & \leq  \OO \left( \frac{N^7}{\epsilon_N^6} \mathbb{P} (\sup d(\hat{\lambda}_n, \mathcal S)>\varepsilon ) \right) \\ & \leq \OO \left( N^{7-0.06 -16} \right) \to 0. \\
\end{aligned}$$
The second condition of tightness is achieved.

For $M_N^{2}(z)$, following exactly the same method in \cite[Section 9.11]{SIL06}, one can show that $M_N^{2}(z)$ is bounded and forms an equicontinuous family that converges to $0$. Hence the tightness for $M(m_{\underline{\hat{\bf R}}_N}(z) -\underline{m}_N(z))$.

\vspace{4mm}
The next step is to prove the tightness of $M(m_{\underline{\hat{\bf R}}_M}'(z) -\underline{m}_N'(z)).$ We have $$ \begin{aligned} & m_{\underline{\hat{\bf R}}_N}'(z_1)-m_{\underline{\hat{\bf R}}_N}'(z_2) \\ &=\frac{z_1-z_2}{M}\sum_{i=1}^{N} \frac{2\hat{\lambda}_i-z_1-z_2}{(\hat{\lambda_i}-z_1)^2(\hat{\lambda}_i-z_2)^2} \\ &=\frac{z_1-z_2}{M} \mathrm{Tr}\left({\bf D}_N^{-2}(z_1) {\bf D}_N^{-2}(z_2)({\bf D}_N(z_1)+{\bf D}_N(z_2))\right).
\end{aligned}$$

Following the same method as derived before, one obtains
$$\begin{aligned} &\frac{\partial}{\partial Y_{ij}} {\bf D}_N^{-1}(z_1) {\bf D}_N^{-2}(z_2) \\ &= -\frac{1}{M} \left[ {\bf Y}_N^{*} {\bf D}_N^{-2}(z_1) {\bf D}_N^{-2}(z_2) + 2 {\bf Y}_N^{*} {\bf D}_N^{-1}(z_1){\bf D}_N^{-3}(z_2) \right]_{j,i}, \end{aligned}$$
and
$$ \left| \frac{\partial}{\partial Y_{ij}}\mathrm{Tr} {\bf D}_N^{-2}(z_1) {\bf D}_N^{-2}(z_2)({\bf D}(z_1)+{\bf D}_N(z_2)) \right|^2 = \frac{1}{M}\bf \mathrm{Tr} ({\bf L}_2)$$
with

$$\begin{aligned} {\bf L}_2=& 4 \hat{\bf R}_N\Big{(}3 {\bf D}_N^{-4}(z_1) {\bf D}_N^{-4}(z_2)+ 2{\bf D}_N^{-3}(z_1) {\bf D}_N^{-5}(z_2) +2{\bf D}_N^{-5}(z_1) {\bf D}_N^{-3}(z_2) \\ &+ {\bf D}_N^{-2}(z_1) {\bf D}_N^{-6}(z_2)+ {\bf D}_N^{-6}(z_1) {\bf D}_N^{-2}(z_2) \Big{)}.\\
\end{aligned}$$
Then Nash-Poincar\'e inequality yields that
$$\begin{aligned} &\textrm{Var}\frac{|{\hat{M}_N^{1'}}(z_1)- {\hat{M}_N^{1'}}(z_2)|}{|z_1-z_2|} \\ & \leq \frac{C_1}{N} \mathbb{E} (\mathrm{Tr}({\bf L}_2) \mathrm{I}_{\sup_n d(\hat{\lambda}_n, \mathcal S)\leq \varepsilon}) + \frac{C_1}{N} \mathbb{E} (\mathrm{Tr}({\bf L}_2) \mathrm{I}_{\sup_n d(\hat{\lambda}_n, \mathcal S)>\varepsilon}) \end{aligned}$$ with $C_1$ the same constant defined as before. The term $\mathrm{Tr}({\bf L}_2)$ is bounded on the set $\sup d(\hat{\lambda}_n,\mathcal S)\leq \varepsilon$. For the second term, $|\mathrm{Tr}({\bf L_2})| \leq \OO \left( \frac{ N^{9}}{\varepsilon_N^{8}} \right)$. As $\mathbb{P}(\sup d(\hat{\lambda}_n, \mathcal S) \geq \varepsilon)=o(N^{-16})$ and $\varepsilon_N=N^{-0.01}$, the proof of the tightness of ${M_N^{1}}'(z)$ is achieved as before.

The proof of the tightness is completed with the verification of ${M_N^{2}}'(z)$ for $z\in \mathcal{C}_n$ to be bounded and forms an equicontinuous family, and convergence to 0.
We will use the same method for the process ${M}_N^{2}(z)$ (see \cite[Section 9.11]{SIL06}).

By Formula (9.11.1) in \cite[Section 9.11]{SIL06}, they show that
\begin{equation}\label{eq:formule}
(\mathbb{E}m_{\underline{\hat{\bf R}}_N} -\underline{m}_N) \left(1- \frac{\frac{N}{M}\int \frac{\underline{m}_N t^2 dF^{\bf{R}_N}(t)}{(1+t\mathbb{E}m_{\underline{\hat{\bf{R}}}_N})(1+t \underline{m}_N)}}{-z+\frac{N}{M}\int \frac{t dF^{\bf{R}_N}}{1+t \mathbb{E}m_{\underline{\hat{\bf{R}}}_N}}- T_N } \right) =\mathbb{E}m_{\underline{\hat{\bf{R}}}_N} \underline{m}_N T_N
\end{equation}
where
$$\begin{aligned} &T_N=\frac{N}{M^2} \sum_{j=1}^M \mathbb{E}\beta_j d_j (\mathbb{E}m_{\underline{\hat{\bf{R}}}_N})^{-1},\\
&d_j=d_j(z)=-{\bf q}_j^{*}{\bf R}^{1/2}(\hat{\bf R}_{(j)}-z{\bf{I}} )^{-1}(\mathbb{E}m_{\underline{\hat{\bf{R}}}_N} {\bf R}+ {\bf{I}})^{-1} {\bf R}^{1/2}{\bf q}_j +(1/M)\mathrm{Tr} (\mathbb{E}m_{\underline{\hat{\bf{R}}}_N} {\bf R} + {\bf{I}})^{-1} {\bf R} (\hat{\bf R}_N-z {\bf{I}})^{-1},\\
&\beta_j=\frac{1}{1+ \frac{1}{M}{\bf y}_j^{*}(\hat{\bf R}_{(j)}-z {\bf{I}})^{-1}{\bf y}_j},\\
&{\bf q}_j=1/\sqrt{N}{\bf x}_j,\\
&\hat{\bf R}_{(j)}= \hat{\bf R}_N- \frac{1}{M}{\bf y}_j {\bf y}_j^{*}.\\
\end{aligned}$$
If one derives (\ref{eq:formule}) with respect to $z$, the equation becomes

$$\begin{aligned} &(\mathbb{E}m'_{\underline{\hat{\bf R}}_N} -\underline{m}'_N) \left(1- \frac{\frac{N}{M}\int \frac{\underline{m}_N t^2 dF^{\bf{R}_N}(t)}{(1+t\mathbb{E}m_{\underline{\hat{\bf{R}}}_N})(1+t \underline{m}_N)}}{-z+\frac{N}{M}\int \frac{t dF^{\bf{R}_N}}{1+t \mathbb{E}m_{\underline{\hat{\bf{R}}}_N}}- T_N} \right) + (\mathbb{E}m_{\underline{\hat{\bf R}}_N} -\underline{m}_N) \left(1- \frac{\frac{N}{M}\int \frac{\underline{m}_N t^2 dF^{\bf{R}_N}(t)}{(1+t\mathbb{E}m_{\underline{\hat{\bf{R}}}_N})(1+t \underline{m}_N)}}{-z+\frac{N}{M}\int \frac{t dF^{\bf{R}_N}}{1+t \mathbb{E}m_{\underline{\hat{\bf{R}}}_N}}-T_N } \right)' \\ &=\mathbb{E}m_{\underline{\hat{\bf{R}}}_N}' \underline{m}_N T_N +\mathbb{E}m_{\underline{\hat{\bf{R}}}_N} \underline{m}_N' T_N +\mathbb{E}m_{\underline{\hat{\bf{R}}}_N} \underline{m}_N T_N'.\\ \end{aligned}$$
In the work of \cite[Section 9.11]{SIL06}, they show that when $N$ tends to infinity,

\begin{enumerate}
\begin{item}
$\sup_{z \in \mathcal{C}_N}|\mathbb{E} m_{\underline{\hat{\bf{R}}}_N}(z)- \underline{m}(z)|\rightarrow 0$ and $\sup_{z\in \mathcal{C}_N}|\underline{m}_N(z)-\underline{m}(z)| \rightarrow 0,$
\end{item}

\begin{item}
$\frac{\frac{N}{M}\int \frac{t^2 \underline{m}_N dF^{\bf{R}_N}(t)}{(1+t\mathbb{E}m_{\underline{\hat{\bf{R}}}_N})(1+t \underline{m}_N)}}{-z+\frac{N}{M}\int \frac{t dF^{\bf{R}_N}}{1+t \mathbb{E}m_{\underline{\hat{\bf{R}}}_N}}-T_N }$ converges ,
\end{item}

\begin{item}
$M_N^{2}(z) \rightarrow 0, \quad T_N \rightarrow 0.$
\end{item}

With the same method, one can show easily that

\begin{item}
$\sup_{z \in \mathcal{C}_N}|\mathbb{E} m_{\underline{\hat{\bf{R}}}_N}'(z)- \underline{m}'(z)|\rightarrow 0, $
\end{item}

\begin{item}
$\sup_{z\in \mathcal{C}_N}|\underline{m}_N'(z)-\underline{m}'(z)| \rightarrow 0,$
\end{item}

\begin{item}
$\frac{N}{M} \left( \sum_{j=1}^M \mathbb{E}\beta_j d_j \right)'$ converges.
\end{item}
\end{enumerate}
With these results, it suffices to show that $T_N' \rightarrow 0,$ and ${M_N^2}'$ is equicontinuous.

In \cite[Section 9.9]{SIL06}, they show that for $m,p \in \mathbb{N}$ and a non-random $N\times N$ matrix ${\bf A}_k$, $k=1,..,m$ and $B_l$, $\ell=1,..,q$, we have

\begin{equation} \label{eq:important1}
\left| \mathbb{E}\left( \prod_{k=1}^{m}{\bf r}_t^{*}{\bf A}_k \bf{r}_t \prod_{\ell=1}^{q}(\bf{r}_t^{*} {\bf R}_{\ell} {\bf r}_t - M^{-1}\mathrm{Tr} {\bf R} {\bf B}_{\ell} ) \right) \right| \leq K M^{-(1\wedge q)} \prod_{k=1}^m \|{\bf A}_k\| \prod_{\ell=1}^q \|{\bf B}_{\ell}\|.
\end{equation}

We have also that for any positive $p$,
\begin{equation} \label{eq:important2}
\max(\mathbb{E}\|{\bf D}^{-1}(z)\|^p, \mathbb{E}\|{\bf D}_j^{-1}(z)\|^p,\mathbb{E}\| {\bf D}_{ij}^{-1}(z)\|^p) \leq K_p
\end{equation}
and
\begin{equation} \label{eq:important3}
\sup_{n,z\in \mathcal{C}_n}\|(\mathbb{E}m_{\underline{\hat{\bf{R}}}_N}(z) {\bf R}+ {\bf I})^{-1} \| < \infty
\end{equation}
where $K_p$ is a constant which depends only on $p$.

With all these preliminaries, as $T_N \rightarrow 0$, by the dominated convergence theorem of derivation, it suffices to show that $T_N'$ is bounded over $\mathcal{C}_N$.
In \cite[Section 9.11]{SIL06}, it is sufficient to show that $(f_M'(z))$ is bounded where

$$f_M(z)=\sum_{j=1}^M \mathbb{E}[({\bf r}_j^{*} {\bf D}_j^{-1} {\bf r}_j - M^{-1}\mathrm{Tr} {\bf D}_j^{-1} {\bf R} )( {\bf r}_j^{*} {\bf D}_j^{-1}(\mathbb{E}m_{\underline{\hat{\bf{R}}}_N} {\bf R}+{\bf I})^{-1} {\bf r}_j-M^{-1}\mathrm{Tr} {\bf D}_j^{-1} (\mathbb{E}m_{\underline{\hat{\bf{R}}}_N} {\bf R}+ {\bf I})^{-1} \bf{R} )].$$
With the help of (\ref{eq:important1})-(\ref{eq:important3}), $f_M'(z)$ is indeed bounded in ${\mathcal C}_N$.

Now we will show that ${M_N^2}'$ is equicontinuous. With the light work as before, it is sufficient to show that $f_M''(z)$ is bounded.
Using (\ref{eq:important1}), we obtain

$$\begin{aligned}
|f''(z)|\leq & KM^{-1} \Big{[} \Big( \mathbb{E}(\mathrm{Tr}{\bf D}_1^{-3} {\bf R} \bar{\bf D}_1^{-3} {\bf R} )\mathbb{E}(\mathrm{Tr}{\bf D}_1^{-1} (\mathbb{E}m_{\underline{\underline{\bf{R}}}_N} {\bf R}+ {\bf I})^{-1}{\bf R}(\mathbb{E}\bar{m}_{\underline{\hat{\bf{R}}}_N} {\bf R}+ {\bf I})^{-1}\bar{\bf D}_1^{-1} {\bf R} ) \Big)^{1/2} \\ &+2 \Big( \mathbb{E}(\mathrm{Tr} {\bf D}_1^{-2} {\bf R} \bar{\bf D}_1^{-2} {\bf R} ) \mathbb{E}(\mathrm{Tr} {\bf D}_1^{-2} (\mathbb{E}m_{\underline{\hat{\bf{R}}}_N} {\bf R}+ {\bf I})^{-1} {\bf R}(\mathbb{E}\bar{m}_{\underline{\hat{\bf{R}}}_N} {\bf R}+ {\bf I})^{-1}\bar{\bf D}_1^{-2} {\bf R} ) \Big)^{1/2}\\ & +2 |\mathbb{E}m_{\underline{\hat{\bf{R}}}_N}'|\Big( \mathbb{E}(\mathrm{Tr} {\bf D}_1^{-2} {\bf R} \bar{\bf D}_1^{-2} {\bf R} )\mathbb{E}(\mathrm{Tr} {\bf D}_1^{-1} (\mathbb{E}m_{\underline{\hat{\bf{R}}}_N} {\bf R}+ {\bf I})^{-2} {\bf R}(\mathbb{E}\bar{m}_{\underline{\hat{\bf{R}}}_N} \bf{R}+ {\bf I})^{-2}\bar{\bf D}_1^{-1}{\bf R} ) \Big)^{1/2}\\ & + \Big( \mathbb{E}(\mathrm{Tr} {\bf D}_1^{-1} {\bf R} \bar{\bf D}_1^{-1} {\bf R} )\mathbb{E}(\mathrm{Tr} {\bf D}_1^{-3} (\mathbb{E}m_{\underline{\hat{\bf{R}}}_N} {\bf R}+ {\bf I})^{-1} {\bf R}(\mathbb{E}\bar{m}_{\underline{\hat{\bf{R}}}_N} {\bf R}+ {\bf I})^{-1}\bar{\bf D}_1^{-3} {\bf R} ) \Big)^{1/2}\\  &+ 2 |\mathbb{E}m_{\underline{\hat{\bf{R}}}_N}'|\Big( \mathbb{E}(\mathrm{Tr} {\bf D}_1^{-1} {\bf R} \bar{\bf D}_1^{-1} {\bf R} )\mathbb{E}(\mathrm{Tr} {\bf D}_1^{-2} (\mathbb{E}m_{\underline{\hat{\bf{R}}}_N} {\bf R}+ {\bf I})^{-2} {\bf R}(\mathbb{E}\bar{m}_{\underline{\hat{\bf{R}}}_N} {\bf R}+ {\bf I})^{-2}\bar{\bf D}_1^{-2} {\bf R} ) \Big)^{1/2} \\ &+|\mathbb{E}m_{\underline{\hat{\bf{R}}}_N}''| \Big( \mathbb{E}(\mathrm{Tr} {\bf D}_1^{-1} {\bf R} \bar{\bf D}_1^{-1} {\bf R} )\mathbb{E}(\mathrm{Tr} {\bf D}_1^{-1} (\mathbb{E}m_{\underline{\hat{\bf{R}}}_N} {\bf R}+ {\bf I})^{-2} {\bf R} (\mathbb{E}\bar{m}_{\underline{\hat{\bf{R}}}_N} {\bf R}+ {\bf I})^{-2}\bar{\bf D}_1^{-1} {\bf R} ) \Big)^{1/2} \\ &+ |\mathbb{E}{m'}_{\underline{\hat{\bf{R}}}_N}|^2 \Big( \mathbb{E}(\mathrm{Tr} {\bf D}_1^{-1} {\bf R} \bar{\bf D}_1^{-1} {\bf R} )\mathbb{E}(\mathrm{Tr} {\bf D}_1^{-1} (\mathbb{E}m_{\underline{ \hat{\bf{R}}}_N} {\bf R}+ {\bf I})^{-3} {\bf R}(\mathbb{E}\bar{m}_{\underline{\hat{\bf{R}}}_N} {\bf R}+ {\bf I})^{-3}\bar{\bf D}_1^{-1} {\bf R} ) \Big)^{1/2} \Big{]}. \\
\end{aligned}$$

Thanks to (\ref{eq:important2}) and (\ref{eq:important3}), the right side is indeed bounded. This ends the proof of the tightness.

\vspace{10mm}

\end{appendix}

\bibliography{IEEEabrv,IEEEconf,tutorial_RMT}

\begin{thebibliography}{10}
\providecommand{\url}[1]{#1}
\csname url@samestyle\endcsname
\providecommand{\newblock}{\relax}
\providecommand{\bibinfo}[2]{#2}
\providecommand{\BIBentrySTDinterwordspacing}{\spaceskip=0pt\relax}
\providecommand{\BIBentryALTinterwordstretchfactor}{4}
\providecommand{\BIBentryALTinterwordspacing}{\spaceskip=\fontdimen2\font plus
\BIBentryALTinterwordstretchfactor\fontdimen3\font minus
  \fontdimen4\font\relax}
\providecommand{\BIBforeignlanguage}[2]{{%
\expandafter\ifx\csname l@#1\endcsname\relax
\typeout{** WARNING: IEEEtran.bst: No hyphenation pattern has been}%
\typeout{** loaded for the language `#1'. Using the pattern for}%
\typeout{** the default language instead.}%
\else
\language=\csname l@#1\endcsname
\fi
#2}}
\providecommand{\BIBdecl}{\relax}
\BIBdecl

\bibitem{PLE02}
V.~Plerous, P.~Gopikrishnan, B.~Rosenow, L.~Amaral, T.~Guhr, and H.~Stanley,
  ``{Random matrix approach to cross correlations in financial data},''
  \emph{Phys. Rev. E}, vol.~65, no.~6, Jun. 2002.

\bibitem{LUO07}
F.~Luo, Y.~Yang, J.~Zhong, H.~Gao, L.~Khan, D.~K. Thompson, and J.~Zhou,
  ``{Constructing gene co-expression networks and predicting functions of
  unknown genes by random matrix theory},'' \emph{BMC bioinformatics}, vol.~8,
  no.~1, p. 299, 2007.

\bibitem{Abla09}
A.~Kammoun, M.~Kharouf, W.~Hachem, and J.~Najim, ``{BER and Outage Probability
  Approximations for LMMSE Detectors on Correlated MIMO channels},''
  \emph{Information Theory}, vol.~55, no.~10, pp. 4386--4397, 2009.

\bibitem{COU10b}
\BIBentryALTinterwordspacing
R.~Couillet, J.~W. Silverstein, and M.~Debbah, ``{Eigen-Inference for Energy
  Estimation of Multiple Sources},'' \emph{{IEEE} Trans. Inf. Theory},
  submitted for publication. [Online]. Available:
  \url{http://arxiv.org/abs/1001.3934}
\BIBentrySTDinterwordspacing

\bibitem{MES08c}
X.~Mestre and M.~Lagunas, ``{Modified Subspace Algorithms for DoA Estimation
  With Large Arrays},'' \emph{{IEEE} Trans. Signal Process.}, vol.~56, no.~2,
  pp. 598--614, Feb. 2008.

\bibitem{SIL06}
Z.~Bai and J.~W. Silverstein, ``{Spectral Analysis of Large Dimensional Random
  Matrices},'' \emph{Springer Series in Statistics}, 2009.

\bibitem{COUbook}
R.~Couillet and M.~Debbah, \emph{{Random matrix methods for wireless
  communications}}, 1st~ed.\hskip 1em plus 0.5em minus 0.4em\relax New York,
  NY, USA: Cambridge University Press, to appear.

\bibitem{GIR00}
\BIBentryALTinterwordspacing
V.~L. Girko, ``{Ten years of general statistical analysis}.'' [Online].
  Available:
  \url{http://www.general-statistical-analysis.girko.freewebspace.com/chapter1%
4.pdf}
\BIBentrySTDinterwordspacing

\bibitem{SIL92}
J.~W. Silverstein and P.~L. Combettes, ``{Signal detection via spectral theory
  of large dimensional random matrices},'' \emph{{IEEE} Trans. Signal
  Process.}, vol.~40, no.~8, pp. 2100--2105, 1992.

\bibitem{KAR08}
N.~E. Karoui, ``{Spectrum estimation for large dimensional covariance matrices
  using random matrix theory},'' \emph{Annals of Statistics}, vol.~36, no.~6,
  pp. 2757--2790, Dec. 2008.

\bibitem{RYA07}
O.~. Ryan and M.~Debbah, ``{Free deconvolution for signal processing
  applications},'' in \emph{Proc. {IEEE} International Symposium on Information
  Theory (ISIT'07)}, Nice, France, Jun. 2007, pp. 1846--1850.

\bibitem{COU08}
R.~Couillet and M.~Debbah, ``{Free deconvolution for OFDM multicell SNR
  detection},'' in \emph{Proc. {IEEE} International Symposium on Personal,
  Indoor and Mobile Radio Communications (PIMRC'08)}, Cannes, France, 2008.

\bibitem{MES08}
X.~Mestre, ``{On the asymptotic behavior of the sample estimates of eigenvalues
  and eigenvectors of covariance matrices},'' \emph{{IEEE} Trans. Signal
  Process.}, vol.~56, no.~11, pp. 5353--5368, Nov. 2008.

\bibitem{MES08b}
------, ``{Improved estimation of eigenvalues of covariance matrices and their
  associated subspaces using their sample estimates},'' \emph{{IEEE} Trans.
  Inf. Theory}, vol.~54, no.~11, pp. 5113--5129, Nov. 2008.

\bibitem{VAL09}
P.~Vallet, P.~Loubaton, and X.~Mestre, ``{Improved subspace DoA estimation
  methods with large arrays: The deterministic signals case},'' in \emph{Proc.
  {IEEE} International Conference on Acoustics, Speech and Signal Processing
  (ICASSP'09)}, 2009, pp. 2137--2140.

\bibitem{MAR67}
V.~A. Mar\u{c}enko and L.~A. Pastur, ``{Distributions of eigenvalues for some
  sets of random matrices},'' \emph{Math USSR-Sbornik}, vol.~1, no.~4, pp.
  457--483, Apr. 1967.

\bibitem{SIL95}
J.~W. Silverstein and Z.~D. Bai, ``{On the empirical distribution of
  eigenvalues of a class of large dimensional random matrices},'' \emph{Journal
  of Multivariate Analysis}, vol.~54, no.~2, pp. 175--192, 1995.

\bibitem{SIL98}
Z.~D. Bai and J.~W. Silverstein, ``{No Eigenvalues Outside the Support of the
  Limiting Spectral Distribution of Large Dimensional Sample Covariance
  Matrices},'' \emph{Annals of Probability}, vol.~26, no.~1, pp. 316--345, Jan.
  1998.

\bibitem{KAL02}
O.~Kallenberg, \emph{{Foundations of mordern Probability, 2nd edition}}.\hskip
  1em plus 0.5em minus 0.4em\relax Springer Verlag New York, 2002.

\bibitem{BAI04}
Z.~D. Bai and J.~W. Silverstein, ``{CLT of linear spectral statistics of large
  dimensional sample covariance matrices},'' \emph{Annals of Probability},
  vol.~32, no.~1A, pp. 553--605, 2004.

\bibitem{Mars87}
J.~Marsden and M.~Hoffman, \emph{{Basic Complex Analysis, 3rd ed}}.\hskip 1em
  plus 0.5em minus 0.4em\relax New York: Freeman, 1987.

\bibitem{BAI99}
Z.~D. Bai and J.~W. Silverstein, ``{Exact Separation of Eigenvalues of Large
  Dimensional Sample Covariance Matrices},'' \emph{The Annals of Probability},
  vol.~27, no.~3, pp. 1536--1555, 1999.

\bibitem{CHO95}
J.~W. Silverstein and S.~Choi, ``{Analysis of the limiting spectral
  distribution of large dimensional random matrices},'' \emph{Journal of
  Multivariate Analysis}, vol.~54, no.~2, pp. 295--309, 1995.

\bibitem{CU89}
F.~Cucker and A.~Corbalan, ``{An Alternate Proof of the Continuity of the Roots
  of a Polynomial},'' \emph{The American Mathematical Monthly}, vol.~96, no.~4,
  pp. 342--345, Apr. 1989.

\bibitem{LOU10}
\BIBentryALTinterwordspacing
P.~Vallet, P.~Loubaton, and X.~Mestre, ``{Improved subspace estimation for
  multivariate observations of high dimension: the deterministic signals
  case},'' \emph{{IEEE} Trans. Inf. Theory}, submitted for publication.
  [Online]. Available: \url{http://arxiv.org/abs/1002.3234}
\BIBentrySTDinterwordspacing

\bibitem{HAC06}
W.~Hachem, O.~Khorunzhy, P.~Loubaton, J.~Najim, and L.~A. Pastur, ``{A new
  approach for capacity analysis of large dimensional multi-antenna
  channels},'' \emph{{IEEE} Trans. Inf. Theory}, vol.~54, no.~9, 2008.

\bibitem{HAA06}
U.~Haagerup, H.~Schultz, and S.~Thorbj{\o}rnsen, ``{A random matrix approach to
  the lack of projections in Cred*(F2)},'' \emph{Advances in Mathematics}, vol.
  204, no.~1, pp. 1--83, 2006.

\bibitem{BIL08}
P.~Billingsley, \emph{{Probability and Measure}}, 3rd~ed.\hskip 1em plus 0.5em
  minus 0.4em\relax Hoboken, NJ: John Wiley \& Sons, Inc., 1995.

\end{thebibliography}
%\bibliography{math}

\end{document}